\documentclass[11pt,a4paper]{amsart}

\usepackage{amssymb,amsmath,verbatim}
\usepackage{latexsym}
\usepackage{eucal}
\usepackage{a4wide} 
\usepackage{amsfonts, bbm}
\usepackage{overpic}

\newtheorem{theorem}{Theorem}[section]
\theoremstyle{plain}

\newtheorem{lemma}[theorem]{Lemma}
\newtheorem{prop}[theorem]{Proposition}
\theoremstyle{remark}

\numberwithin{equation}{section}

\newcommand{\mc}{\mathcal}
\newcommand{\nn}{\mathbb{N}}
\newcommand{\cc}{\mathbb{C}}
\newcommand{\hh}{\mathbb{H}}
\newcommand{\zz}{\mathbb{Z}}

\newcommand{\la}{\lambda}
\newcommand{\eps}{\epsilon}

\newcommand{\pl}{\partial}
\newcommand{\x}{\times}

\newcommand{\til}{\widetilde}
\newcommand{\bbar}{\overline}

\newcommand{\cjd}{\rangle}
\newcommand{\cjg}{\langle}

\newcommand{\demi}{\frac{1}{2}}
\newcommand{\ndemi}{\frac{n}{2}}

\newcommand{\indic}{\mathbbm{1}}

\newcommand{\cA}{\mathcal{A}}
\newcommand{\cB}{\mathcal{B}}

\newcommand{\re}{\operatorname{Re}}
\newcommand{\im}{\operatorname{Im}}
\newcommand{\res}{\operatorname{Res}}
\newcommand{\rank}{\operatorname{rank}}

\newcommand{\supp}{\operatorname{supp}}

\newcommand{\dist}{\operatorname{dist}}

\newcommand{\norm}[1]{\left\Vert #1 \right\Vert}
\newcommand{\brak}[1]{\langle #1 \rangle}
\newcommand{\abs}[1]{\left| #1 \right|}

\newcommand{\bbR}{\mathbb{R}}
\newcommand{\bbH}{\mathbb{H}}
\newcommand{\bbC}{\mathbb{C}}
\newcommand{\bbZ}{\mathbb{Z}}
\newcommand{\bbN}{\mathbb{N}}

\newcommand{\calR}{\mathcal{R}}

\newcommand{\calJ}{\mathcal{J}}
\newcommand{\cinf}{C^\infty}
\newcommand{\del}{\partial}

\newcommand{\vep}{\varepsilon}

\newcommand{\phvert}{\phi^{\rm v}}
\newcommand{\phhor}{\phi^{\rm h}}

\newcommand{\Hyp}{\bbH^{n+1}}

\begin{document}
\title[Resonances on geometrically finite hyperbolic manifolds]{Upper bounds for the number of resonances on geometrically finite hyperbolic manifolds}
\date{Aptil 17, 2013}
\author[Borthwick]{David Borthwick}
\address{Department of Mathematics and Computer Science, Emory
University, Atlanta, Georgia, 30322, USA}
\thanks{DB supported in part by NSF\ grant DMS-0901937.}
\email{davidb@mathcs.emory.edu}
\author[Guillarmou]{Colin Guillarmou}
\address{DMA, U.M.R. 8553 CNRS\\
Ecole Normale Sup\'erieure\\
45 rue d'Ulm\\ 
F 75230 Paris cedex 05 \\France}
\thanks{CG supported in part by grant ANR-09-JCJC-0099-01}
\email{cguillar@dma.ens.fr}

\begin{abstract}
On geometrically finite hyperbolic manifolds $\Gamma\backslash\hh^{d}$, including those with non-maximal rank cusps, we give upper bounds on the number $N(R)$ of resonances of the Laplacian in disks of size $R$ as $R\to \infty$. In particular, if the parabolic subgroups of $\Gamma$ satisfy a certain Diophantine condition, the bound is $N(R)=\mc{O}(R^d (\log R)^{d+1})$.
\end{abstract}

\maketitle

\section{Introduction}

Let $X=\Gamma\backslash\hh^{d}$ be a geometrically finite hyperbolic manifold in the sense of Bowditch \cite{Bowditch:1993}, with the dimension written henceforth as $d=n+1$. We assume that $X$ has infinite volume, in which case the Laplacian $\Delta_X$ on $L^2$ has essential spectrum $[n^2/4,\infty)$, and
a finite number of eigenvalues in $(0,n^2/4]$, as shown by Lax-Phillips \cite{LP:1982}.
The resolvent of $\Delta_X$, 
\[ R_X(s):=(\Delta_X-s(n-s))^{-1},\]
is well defined in $\re s>n/2$, provided $s(n-s)$ is not in the discrete spectrum. It was proved recently by Guillarmou-Mazzeo \cite{GM:2012} that  $R_X(s)$  admits a meromorphic extension to $s\in \cc$ as an 
operator from $L^2_{{\rm comp}}(X)$ to $L_{\rm loc}^2(X)$, with the polar part of the Laurent expansion at any pole 
having finite rank.  (A corresponding result was proved by Bunke-Olbrich \cite{BO:1999b} for the scattering matrix). This continuation was shown previously by Mazzeo-Melrose \cite{MM:1987}
when $X$ has no cusps (i.e.~when the group $\Gamma$ is convex co-compact), by Guillop\'e-Zworski \cite{GZ:1995a} in dimension $2$,
and by Froese-Hislop-Perry \cite{FHP:1991b} in dimension $3$.

The \emph{resonances} of $X$ are the poles of $R_X(s)$ and we denote by $\mc{R}_X$ the set of resonances counted with their multiplicities,  
\[m(s_0):= \rank \res_{s_0}R_X(s),\]  
where ${\rm Res}$ denotes the residue.
Resonances appear naturally in relations with Selberg zeta function and trace formula: in the case of a convex co-compact group $\Gamma$ acting on $\hh^{n+1}$, Patterson-Perry \cite{PP:2001} and Bunke-Olbrich \cite{BO:1999} have shown that the zeros of the meromorphic extension of Selberg zeta function $Z_X(s)$ are given by resonances, modulo topological zeros at negative integers. Direct applications to Selberg type trace formulas can be deduced, see \cite{Perry:2003, GN:2006}. In dimension $2$, building on the work of Guillop\'e-Zworski \cite{GZ:1997}, Borthwick-Judge-Perry \cite{BJP:2005} described the zeros and poles of Selberg function for any geometrically finite hyperbolic surface in terms of resonances.

For all these applications, it was important to know the distribution of resonances in the left half-plane $\re s<n/2$, and in particular upper bounds on their counting function,
\[N_X(R):=\# \left\{ s\in \mc{R}_X;\> |s-n/2|\leq R\right\}.\]
Such growth estimates are in particular crucial for trace formulas and their applications to the counting function of closed geodesics on $X$, see \cite{GZ:1999, GN:2006}. It was shown by Patterson-Perry \cite{PP:2001}, using estimates of Fried on the growth of Selberg zeta function, that 
\[N_X(R)=\mc{O}(R^{n+1})\] 
for convex co-compact $X=\Gamma\backslash\hh^{n+1}$. More generally, 
Borthwick \cite{Borthwick:2008} generalized this estimate to manifolds which are conformally compact with constant curvature near infinity, by directly analyzing   
the parametrix construction of the resolvent by Guillop\'e-Zworski \cite{GZ:1995b} and 
using crucial estimates of Cuevas-Vodev \cite{CV:2003}. In dimension $2$, Guillop\'e-Zworski \cite{GZ:1995a}  
proved $N_X(R)=\mc{O}(R^2)$ for compact perturbations of geometrically finite hyperbolic surfaces. These estimates are sharp as there are also lower bounds of the same order, see \cite{GZ:1999,Perry:2003, Borthwick:2008,BCHP}.
For compact perturbations of Euclidean Laplacian in dimension $d$, sharp upper bounds for resonances of 
the form $N(R)=\mc{O}(R^{d})$ were proved by Zworski \cite{Zworski:1989b}, 
Vodev \cite{Vodev:1991}, Sj\"ostrand-Zworski \cite{SZ:1991}.
 
\bigbreak
In order to state our result, we need to introduce a quantity related to the holonomy of the cusps. Let $X=\Gamma\backslash \hh^{n+1}$ be a geometrically finite hyperbolic manifold, which we assume smooth and oriented (which can always be achieved by passing to a finite cover). Each cusp of rank $k\in[1,n]$ of $X$ can be viewed 
as an open set  in a model quotient manifold $\Gamma^c\backslash\hh^{n+1}$, where 
$\hh^{n+1}$ is represented as a half-space $\bbR^+\x \bbR^n$, and $\Gamma^c$ is an elementary 
parabolic subgroup of ${\rm SO}(n+1,1)$, which is conjugate to a subgroup of $\Gamma$ fixing a point $p$ in the boundary 
$S^n=\pl\hh^{n+1}$  of hyperbolic space. After passing to a finite cover, the group $\Gamma^c$ can be assumed to be abelian and generated by $k$ elements $\gamma_1,\dots,\gamma_k$ of the form
\[ \gamma_j(x,y,z)=(x, A_jy, z+v_j), \quad A_j\in {\rm SO}(n-k), \,\, v_j\in\bbR^k\]
on a certain decomposition $\bbR^+_x\x\bbR_y^{n-k}\x \bbR_z^k$ of the half-space model  of $\hh^{n+1}$.
Let us denote by $\{v_1^*,\dots, v_k^*\}$ a basis for the lattice $\Lambda^*$ dual to  
$\Lambda=\{\sum_{j=1}^ka_jv_j; a_i\in\zz\}$. The group $\Gamma^c$ acts by Euclidean isometries on the horosphere 
$\{x=1\}\simeq \bbR^{n-k}_y\x \bbR_z^k$, and the quotient is a flat vector bundle $F$. The associated spherical bundle 
$SF=\Gamma^c\backslash (S^{n-k-1}\x \bbR^k)$ is well defined because $A_j\in {\rm SO}(n-k)$, and there is a unitary representation $\sigma: \Gamma^c\to {O}(H_m)$ given by $\sigma(\gamma_j)f:=f\circ A_j^{-1}$, where 
$H_m$ is the space of spherical harmonics of degree $m\in\nn_0$ and $O(H_m)$ the orthogonal group of the Hermitian vector space $H_m$ equipped with the $L^2(S^{n-k-1})$ scalar product. Denote by $e^{i\alpha_{mpj}}$ 
(for $p=1,\dots,\mu_m$) the $\mu_m$ eigenvalues of $\sigma(\gamma_j)$ with $\alpha_{mpj}\in[0,2\pi)$, we call these 
$\alpha_{mpj}$ \emph{holonomy angles} of the cusp.  For each $I=(m,p,v^*)\in \nn\x \nn\x \Lambda^*$ with $p\leq \mu_m$, define 
\[
b_I := \biggl| \sum_{j=1}^{k} \alpha_{mpj}v_j^* + 2\pi v^*\biggr|.
\]
Our estimate on the counting function $N_X(R)$ of resonances for $X$ is related to how close 
$b_I$ can approach $0$ when $I=(m,p,v^*)$ satisfies $m\leq R$ but $b_I\not=0$. 
More precisely, we define $\mc{I}_>:=\{I=(m,p,v^*);\> b_I\not=0\}$ and the function
\begin{equation}\label{Lambda.def0}
\Lambda_{\Gamma^c}(u) := 2 \brak{u} \log \brak{u} + 
\sup\limits_{I\in \mc{I}_>, \> 1\le m \le \abs{u}}  \left[ 2(\abs{u}-m) \log \frac{1}{b_I} - 2m \log m \right] .
\end{equation}
We say that $\Gamma^c$ satisfies the \emph{Diophantine condition} if  for some $c>0$, 
$\gamma\ge 0$, 
\[
b_I > cm^{-\gamma},\>\text{ for }I \in \mc{I}_>.
\]
Under this condition, $\Lambda_{\Gamma^c}(u)=\mc{O}(\cjg u\cjd \log\cjg u\cjd)$. We are able to prove:
\begin{theorem}\label{maintheorem}
Let $X=\Gamma\backslash \hh^{n+1}$ be a geometrically finite hyperbolic manifold with $n_c$ cusps and let $\mc{R}_X$ be the set of resonances with multiplicity for its Laplacian. Let $\Gamma^c_1,\dots,\Gamma^c_{n_c}$ 
be $n_c$ parabolic subgroups associated to each cusp and define the function 
$\Lambda_{X}(u):=\max_{j\leq n_c}\Lambda_{\Gamma^c_j}(u)$ using \eqref{Lambda.def0}. 
Then there exists $C$ such that for all $R>1$,
\[ N_X(R)\leq \frac{C (\Lambda_X(2R))^{n+2}}{R}.\] 
In particular, if the cusps all satisfy the Diophantine condition, 
\[ N_X(R)\leq CR^{n+1} (\log R)^{n+2}.\]
\end{theorem}

The Diophantine condition is obviously satisfied when the holonomy angles $\alpha_{mpj}/2\pi$ are rational, and it continues to hold when these angles are algebraic.  In \S \ref{Dioph.sec}, we give examples with transcendental angles where $\Lambda_{X}(R)$ can grow arbitrarily fast.  In such cases the estimate of Theorem~\ref{maintheorem} is not very good, but it is not clear that it could be improved. The growth in the estimate comes from the fact that 
the model resolvent for $\Gamma^c\backslash \hh^{n+1}$ has very large norm when $b_I\to 0$ fast as $m\to \infty$, if $\Gamma^c$ is a parabolic subgroup with $\Lambda_{\Gamma^c}(R)$ large - see Proposition \ref{RXc.estimate}.  Strangely enough, the model space
$\Gamma^c\backslash \hh^{n+1}$ has only $\mc{O}(R^{n-k+1})$ resonances in a ball of radius $R$.  If we were to perturb this model case inside a compact set, it is conceivable that a large number of resonances would appear due to the large norm of the model resolvent. 

The proof of Theorem~\ref{maintheorem} is based on a precise parametrix construction, inspired in parts by techniques of \cite{GZ:1995b,CV:2003} and 
the work of \cite{GM:2012}.  As usual with such methods, we could allow $X$ to be a smooth, compactly supported metric perturbation of $\Gamma\backslash \hh^{n+1}$, or a more generally a manifold with neighbourhoods of infinity isometric to model neighbourhoods from a geometrically finite hyperbolic quotient.  These generalizations would not change the proof, but we restrict our attention to $\Gamma\backslash \hh^{n+1}$ for simplicity of exposition.

We note that our method gives an alternate, simplified proof of the sharp estimate $N_X(R)=\mc{O}(R^{n+1})$ for conformally compact manifolds with constant curvature near infinity, first proved entirely in \cite{Borthwick:2008}.  In particular, we are able to understand more precisely than in \cite{GZ:1995b,CV:2003} the multiplicity of poles in the model terms used for the parametrix.  This issue was a crucial reason why the bound in \cite{GZ:1995b} was not optimal, and also the reason for the omission of a sector containing the negative real axis from the resonance count in \cite{CV:2003}.  In the proof for the conformally compact case in \cite{Borthwick:2008}, the model-term multiplicity issue was bypassed by counting resonances in the missing sector as zeros of a regularized determinant of the scattering matrix.  This isn't possible in the case of non-maximal rank cusps, because the scattering matrix, whose existence was demonstrated in \cite{GM:2012}, naturally acts on an ideal boundary manifold that is not compact.  It therefore proves very difficult to produce a regularized scattering determinant whose zeros still correspond to resonances.  Improved control of multiplicities in the model terms of the parametrix construction is thus an essential feature of our argument.

The bound $N_X(R)=\mc{O}(R^{n+1})$ extends also to the case of geometrically hyperbolic manifolds whose cusps all have maximal rank.
(This hasn't been formally written down for $n\ge 2$, but the methods of \cite{GZ:1995a} clearly generalize to maximal rank cusps in higher dimensions.) 
For cases with cusps of non-maximal rank, even assuming the Diophantine condition or the rational  assumption on the holonomy angles $\alpha_{mpj}/2\pi$, it is not clear if the $R^{n+1}(\log R)^{n+2}$ bound could be improved to $\mc{O}(R^{n+1})$. In principle our bound should imply upper bounds on the number of zeros of Selberg zeta function in that setting, 
but the meromorphic extension of this function and the analysis of its zeros are not done yet. It would be interesting to see, by Fried's method using transfer operators, if better estimates can be obtained on the growth of Selberg function in the complex plane, provided it has meromorphic extension. 

The methods developed here also yield a bound for the resonance counting function in a vertical strip.  The resulting estimate,
\[
\#\bigl\{s \in \mc{R}_X;\> \re s > n/2-K,\>0 \le \im s \le T  \bigr\} \le C_K T^{n+2}
\]
(see Proposition~\ref{Nstrip.prop}), is not sharp.  It does, however, show that we can control the resonance count in strips without reference to the Diophantine approximation problem that affects the global estimate.  

The paper is organized as follows.  In \S\ref{setup.sec} we review the structure of model neighbourhoods of infinity for a geometrically finite quotient and cover some technical preliminaries.   The parametrix constructions for the model neighbourhoods of infinity are presented for the regular case in \S\ref{par.reg.sec} and for the cusp case in \S\ref{par.cusp.sec}, and precise estimates for the singular values of these parametrix terms are developed.  In \S\ref{gl.par.sec} we apply these singular value estimates to produce the determinant growth estimate that leads to the proof of Theorem~\ref{maintheorem}.  Some technical details of the parametrix construction and special function estimates are relegated to the Appendix.

\bigbreak
\section{Geometry and setup}\label{setup.sec}

Generally we will use the upper half-plane model of the hyperbolic space $\hh^{n+1}$, given by $\bbR_x^+\x\bbR_y^{n}$ with the metric 
\[
g_0 = \frac{dx^2+|dy|^2}{x^2}.
\] 
However, it is sometimes useful to view $\hh^{n+1}$ as the interior of a compact manifold with boundary, namely the closed unit ball $\{m\in\bbR^{n+1}; |m|\leq 1\}$, equipped with the complete metric  $g_{0}=4|dm|^2/(1-|m|^2)^2$. 
We let $\bbar{\hh}^{n+1}$ denote the closed unit ball compactification of $\hh^{n+1}$ defined by this identification.

Let $X=\Gamma\backslash \hh^{n+1}$ be a geometrically finite quotient. 
We follow the geometric description of Mazzeo-Phillips \cite{MP:1990}  and Guillarmou-Mazzeo \cite[\S2]{GM:2012}; see also Bowditch \cite{Bowditch:1993}.
After passing to a finite cover (which does not affect resonance 
counting), $X$ can be covered 
by three types of open sets: the cusp neighbourhoods 
$\{\mc{U}_j\}_{j\in J^c}$, the regular neighbourhoods $\{\mc{U}_j\}_{j\in J^r}$, 
and a relatively compact interior set $\mc{U}_0$.  Here $J^c\cup J^r\subset \nn$ are index sets with $J^r\cap J^c=\emptyset$. 
The cusp neighbourhoods can be taken disjoint from each other, by adding regular open sets if necessary, so that we assume $\mc{U}_j\cap \mc{U}_{\ell}=\emptyset$ if $(j,\ell)\in J^c\x J^c, j\not=\ell$.
Each regular neighbourhood $\mc{U}_j$ with $j\in J^r$ is isometric to a half-ball in the half-space model 
\begin{equation}\label{halfball} 
\mc{U}_j\simeq B_0:= \left\{(x,y)\in \hh^{n+1} ; x^2+|y|^2< 1\right\}, \textrm{ with metric }Ê\frac{dx^2+dy^2}{x^2}.
\end{equation}
Each cusp neighbourhood $\mc{U}_j$ with $j\in J^c$ is isometric to a subset 
\begin{equation}\label{cuspmodel}
\mc{U}_j\simeq \Gamma^c_j\backslash \left\{ (x,y,z)\in\hh^{n+1};  x^2+|y|^2>R_j \right\},Ê\textrm{ with metric }\frac{dx^2+dy^2+dz^2}{x^2} 
\end{equation}
for some $R_j>0$,  where $\Gamma_j^c$ is an abelian elementary parabolic group with rank $k_j\in[1,n]$ 
fixing $\infty$ in the half-space model of $\hh^{n+1}$ and acting as translations in the $z\in \bbR^{k_j}$ variable.  More precisely, 
$\Gamma_j^c$ is the abelian group generated by some $\gamma_1,\dots, \gamma_{k_j}$ which acts on $\hh^{n+1}=\bbR^+\x \bbR^{n-k_j}\x \bbR^{k_j}$ by
\begin{equation}\label{gen.form}
 \gamma_\ell(x,y,z)= (x, A_{\gamma_\ell}y, z+v_{\gamma_\ell})
\end{equation}
for some $v_{\gamma_\ell}\in \bbR^{k_j}$ and some commuting family $A_{\gamma_\ell}\in SO(n-k_j)$, $\ell=1,\dots,k_j$.

The cusp neighbourhood $\mc{U}_j$, $j\in J^c$, can also be viewed as the region $\{x^2+|y|^2>R_j\}$ in the quotient
$\Gamma_j^c\backslash \hh^{n+1}$, where the functions $x,|y|^2$ on $\hh^{n+1}$ descend 
to smooth functions on the quotient. There is an isometry,
\[ \Gamma_j^c \backslash \hh^{n+1}\simeq \Big(\bbR^+_x\x F_j ,  \frac{dx^2+ g_{F_j}}{x^2}\Big), \]
where $F_j$ is a flat vector bundle over a flat $k_j$-dimensional torus obtained by the quotient
$F_j:= \Gamma_j^c\backslash \bbR^{n}$, with action $\gamma_\ell(y,z)=(A_{\gamma_\ell}y, z+v_{\gamma_\ell})$ 
induced from the action of $\gamma_\ell$ on the horospheres $x={\rm const}$.  Since 
the action is by Euclidean isometry the quotient $F_j$ inherits a flat metric $g_{F_j}$. 

We will also use the space $\bbar{X}:=\Gamma\backslash (\bbar{\hh}^{n+1}\setminus \Lambda(\Gamma))$
where $\Lambda(\Gamma)\subset S^n=\pl\bbar{\hh}^{n+1}$ is the limit set of the group $\Gamma$. The space 
$\bbar{X}$ is a smooth manifold with boundary $\pl\bbar{X}$, and this boundary is given locally by $\{x=0\}$ 
in the charts $\mc{U}_j$ for $j\in J^c\cup J^r$.  In the conformally compact case, $\bbar{X}$ would be the conformal compactification of $X$, but in cases with non-maximal rank cusps, neither $\bbar{X}$ nor $\pl\bbar{X}$ are compact. \\ 

For simplicity of notation,  in what follows we will identify the neighbourhoods $\mc{U}_j$ for $j\in J^c,J^r$ with the respective models \eqref{halfball} and \eqref{cuspmodel}. We will also assume that there is no cusp of maximal rank 
$n$, as  their analysis is now standard, see for instance \cite{GZ:1995a} in dimension $2$.\\

\subsection{Notations and parameters}

There will be several parameters in the construction: $s\in \cc$ is the spectral parameter which we sometimes
change to $\la=s-n/2$ to simplify certain notations; $N$ is a large parameter indexing the parametrix construction, such that the parametrix yields continuation of the resolvent to $\re s>n/2-N$;
$\delta>0$ is a small parameter, independent of $s,N$, 
used to localize certain parametrix terms in small neighbourhoods of infinity, in a way similar to \cite{GZ:1995b}. 
The need for this localization will become clear in Propostions~\ref{GZ.prop}.  Certain error terms which would otherwise be exponentially large can be bounded for $\delta$ sufficiently small. This parameter does not appear for instance 
in \cite{GM:2012}, and is only used here for the singular value estimates of remainders in the parametrix constructions, which are of course the core of the bound on the resonances counting function. 

We use $C$ to denote a generic positive constant whose value can change from line to line, and which does not depend on the parameters $s,N,\delta$. Finally $\nn$ denotes the set of positive integers and $\nn_0=\nn\cup\{0\}.$

\subsection{Weight function depending on $\delta$.} 
For technical reasons we also need a parameter-dependent weight function, denoted
$\rho\in L^\infty(\bbar{X})$ such that $0<\rho\le 1$ in $X$.  Near the boundary we want $\rho$ to be comparable to the variable $x$ under each 
identification of a boundary neighbourhood with the models \eqref{halfball} and \eqref{cuspmodel}.  But for the small parameter
$\delta>0$ described above, we want to have $\rho = 1$ for $x \ge \delta$ in each boundary neighbourhood.   The purpose of choosing the weight function this way is to prevent the localization of boundary parametrix terms described above from adversely affecting estimates of the interior parametrix term.

Because the boundary neighbourhoods overlap, and each has a different $x$ coordinate, we need to demonstrate the existence of a function with the desired properties.
Let $(\mc{U}_j)_{j\in J^r\cup J^c}$ be a finite covering of a neighbourhood of $\pl\bbar{X}$ as described above.
For $j\in J^{r}\cup J^c$, denote by $x_j$ the pull-back of the function $x$ to $\mc{U}_j$ through the isometry 
$\psi_j$ identifying $\mc{U}_j$ with $B_0$ or
$\Gamma^c_j\backslash \left\{ (x,y,z)\in\hh^{n+1};  x^2+|y|^2>R_j \right\}$, as described in \eqref{halfball} and \eqref{cuspmodel}. 

\begin{lemma}\label{constrho} 
There exists a positive function $\rho\in L^\infty(X)$ with $\rho \le 1$ everywhere and a constant $C$ independent of $\delta$ such that
for all $i\in J^{r}\cup J^c$,
\begin{equation}\label{condrho}
\begin{cases}x_i\leq \rho\leq \frac{C}{\delta}x_i &\textrm{in }\mc{U}_i\cap\{x_i\leq \delta\},\\
\rho =1&\textrm{in }\mc{U}_i\cap\{x_i\geq \delta\}.\end{cases}
\end{equation}
\end{lemma}
\begin{proof}
First, there is a constant $M\geq 1$ such that for all $i,j\in J:=J^c\cup J^r$, 
\begin{equation}\label{xMMx}
\frac{x_j}{M} \leq x_{i}\leq Mx_j,  \quad \textrm{Êin }\mc{U}_i\cap \mc{U}_j.
\end{equation}
Taking $\delta$ small enough, we can always assume that $\delta M^2\leq 1$.
We introduce a partition of unity $1=\chi_0+\sum_{j=1}^J \chi_j$ with $1-\chi_0=1$ in a region containing 
$\cup_{j\in J}\{x_j\leq 1/4\}$ and $\chi_j$ supported in $\mc{U}_j$, and then set
\[\rho:= \chi_0+ \sum_{j=1}^J (Mx_j \indic_{x_j\leq \delta/M}+ \indic_{x_j>\delta/M}) \chi_j .\]
Clearly $\rho\leq 1$.  From \eqref{xMMx} we see that in $\mc{U}_i\cap \mc{U}_j$,
\[
\frac{1}{x_i}\indic_{x_j>\delta/M} \le \frac{M}{x_j}\indic_{x_j>\delta/M} \le M^2/\delta.
\]
It follows easily that $\rho/x_i \le M^2/\delta$ in $\mc{U}_i$.  The gives in particular the upper bound in the first line of \eqref{condrho}.

In $\mc{U}_i$, if $x_i\geq \delta$, then  in any $\mc{U}_i \cap \mc{U}_j$ we also have $x_j\geq \delta/M$.  It follows that $\rho = 1$ in $\mc{U}_i \cap\{x_i \ge \delta\}$.
On the other hand, if $x_i\leq \delta$, then we also
we have that $M x_j\leq M^2x_i\leq M^2\delta\leq 1$.  Thus, in $\mc{U}_i \cap \{ x_i\le\delta\}$,
\[\rho\geq \chi_0+\sum_{j}Mx_j\chi_j \geq \chi_0+x_i \sum_{j}\chi_j\geq x_i.\]
This completes the proof.\end{proof}

\bigbreak
\section{Parametrix construction for regular neighbourhoods}\label{par.reg.sec}

To construct the resolvent on a geometrically finite hyperbolic manifold $X$, we will construct a parametrix using the fact that for most points $p$ near infinity of $X$, the geometry of $X$ near $p$ is isometric to $B_0$. 
For this reason, we begin with the construction of the parametrix in a model neighbourhood of the regular type \eqref{halfball}, i.e. a half-ball $B_0$ in $\hh^{n+1}$.  This case was of course considered in Guillop\'e-Zworski \cite{GZ:1995b}, but as noted in the remarks following the statement of Theorem~\ref{maintheorem}, we need a modified construction that allows more precise control of the multiplicities appearing in the model terms.

\subsection{Resolvent on the covering space $\hh^{n+1}$}

The Laplacian $\Delta_{g_0}$ in the upper half-space model is 
\[\Delta_{g_0}=-(x\pl_x)^2+nx\pl_x+x^2\Delta_y.\]
Its resolvent, $R_0(s)=(\Delta_{g_0}-s(n-s))^{-1}$, 
is well-defined as a bounded operator on $L^2(\hh^{n+1})$ for $\re s>n/2$.  Moreover, it
has a Schwartz kernel which is explicit, given in terms of the hypergeometric function ${_{2}F}_1$ (see \cite{Patterson:1975}):
\begin{equation}\label{R0.sigma}
R_{0}(s;w,w') = \pi^{-n/2} 2^{-2s-1} \frac{\Gamma(s)}{\Gamma(s - n/2+1)} \sigma^{-s}
{_2F}_1\left(s,s-\tfrac{n-1}2; 2s-n+1; \sigma^{-1} \right),
\end{equation}
where
\begin{equation}\label{sigma}
\sigma := \cosh^2(\demi d(w,w')) = \frac{(x+x')^2 + |y-y'|^2}{4xx'}.
\end{equation}
For $\re s>(n-1)/2$ we can write this, by Euler's integral formula, as
\begin{equation}\label{hypergeom}
R_{0}(s;w,w') = 
\frac{\pi^{-(n+1)/2} 2^{-n-1} \Gamma(s)}{\Gamma(s-\frac{n-1}2)}
\int_{0}^1 \frac{(t(1-t))^{s-\frac{n-1}2}}{(\sigma-t)^{s}}dt.
\end{equation}
An alternative expansion from \cite[Lemma~2.1]{GZ:1995b} is
\begin{equation}\label{R0.tau}
R_{0}(s;w,w') = \pi^{-n/2} 2^{-s-1} \sum_{j=0}^\infty 2^{-2j} \frac{\Gamma(s+2j)}{\Gamma(s-n/2+1+j) \Gamma(j+1)}
\tau^{-s-2j},
\end{equation}
where
\begin{equation}\label{tau}
\tau := \cosh d(w,w') =  \frac{x^2+{x'}^2 + |y-y'|^2}{2xx'}.
\end{equation}
The expressions \eqref{R0.sigma}, \eqref{R0.tau} extend to $s\in \cc$ and produce a meromorphic family of continuous 
operators mapping $L^2_{\rm comp}(\hh^{n+1})$ to $L^2_{\rm loc}(\hh^{n+1})$.  The resolvent $R_0(s)$ also extends to a continuous map $\dot{C}^\infty(\bbar{\hh}^{n+1})\to x^{s}C^\infty(\bbar{\hh}^{n+1})$, where 
$\dot{C}^\infty(\bbar{\hh}^{n+1})$ is the space of smooth functions on $\bbar{\hh}^{n+1}$ which vanish 
to infinite order at the boundary $\pl\bbar{\hh}^{n+1} = S^n$ and $x$ is any smooth boundary defining function of the boundary $\pl\bbar{\hh}^{n+1}$.

The family $R_0(s)$ is analytic for all $s\in\bbC$ when $n$ is even, while for $n$ odd it has simple poles at $-\nn_0$ with the residue at $-k\in-\nn_0$ given by the finite rank operator with Schwartz kernel,
\begin{equation}\label{Residue}
\res_{s=-k}(R_0(s))=
\sum_{0\leq 2j\leq k} \frac{\pi^{-\ndemi}(-1)^{k+2j}2^{k-2j-1}}{j! (k-2j)!\Gamma(j-\ndemi+1-k)}\cosh^{k-2j}(d(w,w')).
\end{equation}
The rank of this operator is computed in \cite[Appendix]{GZ:1995a}: 
\begin{equation}\label{rankRes}
\rank\, \res_{s=-k}(R_0(s))=\dim \ker (\Delta_{S^{n+1}}-k(k+n))=\mc{O}((1+k)^n).
\end{equation} 
We shall need a slightly more precise statement for what follows:  
\begin{lemma}\label{expleft}
Consider the half-space model $\bbR^+_x\x\bbR^n_y$ of $\hh^{n+1}$ and let $\delta>0$. 
Then there exist operators  $M_{\ell}(s):C_0^\infty(\hh^{n+1})\to C^\infty(\bbR^{n})$ and 
$R_{0,N}(s): C_0^\infty(\hh^{n+1})\to x^{s+2N}C^\infty([0,\delta)\x\bbR^n)$, for $\ell, N\in \nn$, such that 
for any $\chi\in C_0^\infty([0,\delta)\x \bbR^n)$ and $\varphi\in C_0^\infty(\hh^{n+1})$, 
\[ (\chi R_0(s)\varphi)(x,y) =\chi(x,y) \sum_{\ell=0}^{N-1} x^{s+2\ell}(M_\ell(s)\varphi)(y)+
(\chi R_{0,N}(s)\varphi)(x,y).\]
In addition, $\Gamma(s-\ndemi+\ell+1)M_\ell(s)$ is meromorphic in $\{\re s>n/2-N\}$ with at most simple poles
at $-k\in -\nn_0\cap\{\re s>n/2-N\}$, and with residue of the form 
\begin{equation}\label{Ress=-k}
\res_{s=-k}(\Gamma(s-\tfrac{n}{2}+\ell+1)M_\ell(s))=\sum_{i=1}^{J(k)}u^{(\ell)}_{k,i}\otimes v_{k,i}
\end{equation}
for some $u^{(\ell)}_{k,i} \in C^\infty(\bbR^n)$, $v_{k,i}\in C^\infty(\hh^{n+1})$, with $J(k)=\mc{O}(k^n)$ for $k$ large. 
The operator $R_{0,N}(s)$ is meromorphic in $\{\re s>n/2-N\}$, with simple poles at $-k\in\nn_0$ and 
residue of rank $\mc{O}(k^n)$.
\end{lemma}
\begin{proof}
Let  $\chi\in C_0^\infty([0,\infty)\x \bbR^n)$ and $\varphi\in C_0^\infty(\hh^{n+1})$. Using \eqref{R0.tau}, the leading term 
of $(R_0(s)\varphi)(x,y)$ at $x=0$ is given by
\[ [x^{-s}(R_0(s)\varphi)(x,y)]|_{x=0} =  \frac{\pi^{-\ndemi}2^{s-1}
\Gamma(s)}{\Gamma(s-\ndemi+1)}\int_{\hh^{n+1}}(x'^2+|y-y'|^2)^{-s}{x'}^{s}\varphi(x',y')\frac{dx'dy'}{{x'}^{n+1}}.\]
Since the Laplacian is given by $\Delta_{g_0}=-(x\pl_x)^2+nx\pl_x+x^2\Delta_y$ in these coordinates, 
and since we know that $(R_0(s)\varphi)\in x^{s}C^\infty([0,\infty)\x\bbR^n)$, the Taylor expansion of the solution 
of $(\Delta_{g_0}-s(n-s))(R_0(s)\varphi)=0$  near $x=0$ ($\varphi$ has compact support in $\hh^{n+1}$)
is formally determined by the equation and is given modulo $\mc{O}(x^{\re s+2N+2})$ by 
\[ (R_0(s)\varphi)(x,y)\sim\pi^{-\ndemi}\sum_{\ell=0}^{N-1} \frac{2^{-2\ell+s-1}
\Gamma(s)}{\Gamma(s-\ndemi+\ell+1)}x^{s+2\ell}
\Delta_y^{\ell} \int_{\hh^{n+1}}(x'^2+|y-y'|^2)^{-s}{x'}^{s}\varphi(x',y')\frac{dx'dy'}{{x'}^{n+1}}.
\]
Therefore, the operators $M_\ell(s)$ we are looking for are 
\begin{equation}\label{Mell}
(M_\ell(s)\varphi)(y):=\pi^{-\ndemi} \frac{2^{-2\ell+s-1}
\Gamma(s)}{\Gamma(s-\ndemi+\ell+1)}
\Delta_y^{\ell} \int_{\hh^{n+1}}(x'^2+|y-y'|^2)^{-s}{x'}^{s}\varphi(x',y')\frac{dx'dy'}{{x'}^{n+1}}.
\end{equation}

The poles of $\Gamma(s-\ndemi+\ell+1)M_\ell(s)$ are simple and located at $-k\in-\nn_0$, and the residue is the operator with Schwartz kernel
\[\Big(\res_{s=-k} \Gamma(s-\tfrac{n}{2}+\ell+1)M_\ell(s)\Big)(y,x',y')=C_{k,\ell}
\Delta_y^{\ell} (x'^2+|y-y'|^2)^{k}{x'}^{-k}\]
for some constant $C_{k,\ell}\in\cc$.   We can expand
\[(x'^2+|y-y'|^2)^{k}{x'}^{-k}=\sum_{\alpha\in\nn_0^n, \abs{\alpha}\leq 2k}  C_{k,\alpha}\>
y_1^{\alpha_1}\dots y_n^{\alpha_n} v_{k,\alpha}(x',y')\] 
for some $C_{k,\alpha} \in\cc$ and $v_{k,\alpha}\in C^\infty(\hh^{n+1})$.  
Thus we see that the residue of 
$\Gamma(s-\tfrac{n}{2}+\ell+1)M_\ell(s)$ has finite rank and is of the form 
\begin{equation}\label{ress=-k} 
\res_{s=-k}(\Gamma(s-\tfrac{n}{2}+\ell+1)M_\ell(s))=
\sum_{\alpha\in\nn_0^n, \abs{\alpha}\leq 2k}  u^{(\ell)}_{k,\alpha} \otimes v_{k,\alpha},
\end{equation}
where for some constants $C_{k,\ell,\alpha}$
\[
u^{(l)}_{k,\alpha}(y) = C_{k,\ell,\alpha}\>\Delta_y^{\ell} (y_1^{\alpha_1}\dots y_n^{\alpha_n}).
\]
Hence \eqref{Ress=-k} can be satisfied with $J(k)$ given by the dimension of the space of monomials of degree at most $2k$ in $n$ variables, which is  $\mc{O}(k^n)$. 
 
The rank estimate for the residues of $R_{0,N}(s)$ is a direct consequence of the rank estimate for $M_{\ell}(s)$,
the fact that $R_0(s): C_0^\infty(\hh^{n+1})\to x^{s}C^\infty(\bbar{\hh}^{n+1})$, and the model resolvent rank estimate
\eqref{rankRes}.   This completes the proof. 
\end{proof}

\subsection{Parametrix construction} \label{paramreg}
Consider the half-ball $B_0 \subset \bbH^{n+1}$ introduced in \eqref{halfball}.
We define its partial closure in $\bbar{\hh}^{n+1}$ by
\[ \bbar{B}_0:=\{(x,y)\in [0,\infty)\x \bbR^{n} ; \>x^2+y^2<1\},\] 
and set $\pl\bbar{B}_0:=\bbar{B}_0\cap \{x=0\}\subset \bbR^{n}$.
 
The main cutoff for the model neighbourhood is $\chi_\delta\in C_0^\infty(\bbar{B}_0)$.  This comes from the partition of unity for the cover introduced in \S\ref{setup.sec}.  For technical reasons, we assume that $\chi_\delta$ is supported in $\{x\leq \delta\}$.
We also introduce an outer cutoff decomposed into horizontal and vertical components as $\phvert_\delta \phhor$.
Here $\phvert_\delta = \phvert_\delta(x) \in C_0^\infty([0,2\sqrt{\delta}))$ with $\phvert_\delta=1$ on $[0, \sqrt{\delta}]$. The horizontal component $\phhor\in C^\infty_0(\pl\bbar{B}_0)$, is independent of $\delta$, and 
chosen so that $\phvert_\delta\phhor=1$ on the support of $\chi_\delta$. 
The structure of these cutoffs is illustrated in Figure~\ref{regcutoffs.fig}.

For $N$ large, we can further assume that $\phhor,\phvert_\delta$ are chosen, depending on $N$, so as to 
satisfy quasi-analytic estimates.  By this we mean that there exists $C>0$ independent of $N$ such that
for all $\alpha\in\nn^n$ and $j\in\nn$ with $\max(\abs{\alpha},j) \leq 10N$
\begin{equation}\label{QAE}
\norm{\partial_x^j \phvert_\delta}_{L^\infty}\leq (C/\sqrt{\delta})^{j}N^{j}, \quad 
\norm{\partial_y^\alpha \phhor}_{L^\infty}\leq C^{\abs{\alpha}}N^{\abs{\alpha}}.
\end{equation}
The existence of such functions is proved in \cite[Thm.~1.4.2]{Hormander:I}.
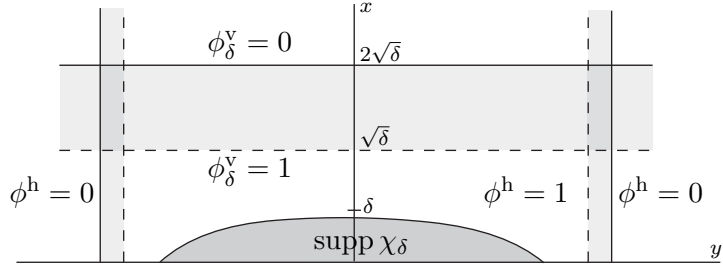
\begin{figure}
\begin{center}
\begin{overpic}{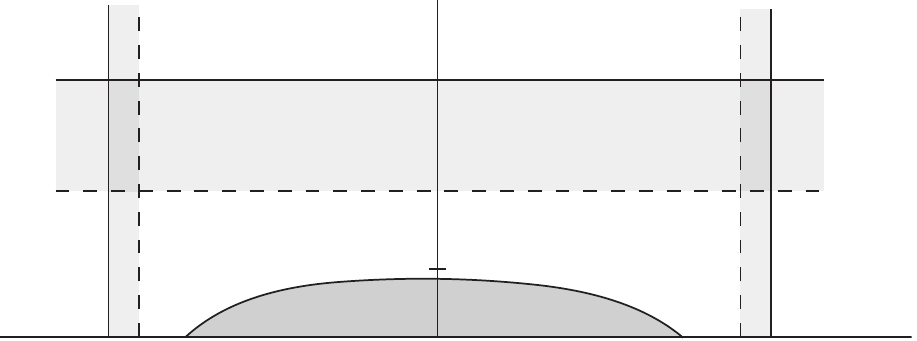}
\put(49,7.4){$\scriptstyle\delta$}
\put(48.5,17){$\scriptstyle\sqrt{\delta}$}
\put(48.5,29){$\scriptstyle2\sqrt{\delta}$}
\put(42,2.2){$\supp \chi_\delta$}
\put(66,9){$\phhor = 1$}
\put(-1,9){$\phhor = 0$}
\put(85,9){$\phhor = 0$}
\put(27,30.5){$\phvert_\delta=0$}
\put(27,12.5){$\phvert_\delta=1$}
\put(98,1.5){$\scriptstyle y$}
\put(48.5,36){$\scriptstyle x$}
\end{overpic}
\end{center}
\caption{Structure of cutoffs in the regular neighbourhood.}\label{regcutoffs.fig}
\end{figure}

For the initial parametrix we simply cutoff the model resolvent to $\phvert_\delta\phhor R_0(s)\chi_\delta$.
Since $\Delta_{g_0}=-(x\pl_x)^2+nx\pl_x+x^2\Delta_y$, one has
\[
(\Delta_{g_0}-s(n-s))\phvert_\delta\phhor R_0(s)\chi_\delta=\chi_\delta+ K_0(s)+L_0(s), 
\]
where
\[ \begin{gathered}
 L_0(s):=\phhor[-(x\pl_x)^2+nx\pl_x,\phvert_\delta]R_0(s)\chi_\delta, \\ 
 K_0(s):=x^2 \phvert_\delta [\Delta_y,\phhor]R_0(s)\chi_\delta.
\end{gathered}\]
From \eqref{R0.tau} and the fact that $\chi_\delta\nabla(\phvert_\delta\phhor)=0$, we see that the Schwartz kernels of $K_0(s)$ and $L_0(s)$ satisfy
\[K_0(s)\in x^{s+2}{x'}^sC^\infty(\bbar{B}_0\x\bbar{B}_0),\quad 
L_0(s)\in x^{\infty}{x'}^sC^\infty(\bbar{B}_0\x\bbar{B}_0).\] 
The $L_0(s)$ term already belongs to any Schatten class as an operator
on $x^NL^2(B_0,dg_0)$ for $\re s>n/2-N$, but the $K_0(s)$ term does not.
Thus we need to pursue the parametrix construction to improve this error. 

The expansion given in Lemma \ref{expleft} implies 
\[K_0(s)=\phvert_\delta [\Delta_y,\phhor]\sum_{\ell=0}^{N-1}x^{s+2\ell+2}M_\ell(s)\chi_\delta+ x^2 \phvert_\delta [\Delta_y,\phhor]R_{0,N}(s)\chi_\delta,
\]
where $M_{\ell}(s)$ are defined by
\eqref{Mell} and the term involving $R_{0,N}(s)$ is in 
$x^{s+2N+2}{x'}^{s}C^\infty(\bbar{B}_0\x\bbar{B}_0)$. 
Applying Lemma \ref{f.abc} to 
\begin{equation}\label{Deffj}
f_j(s;y,w'):= [\Delta_y,\phhor(y)]M_{j-1}(s;y,w')\chi_\delta(w')
\end{equation} 
with $w'$ viewed as a parameter, there exist some differential operators 
$\mc{A}_{j,N}(s), \mc{B}_{j,N}(s)$ with smooth coefficients on $\bbR^+\x\bbR^n$, such that 
\[
(\Delta_{g_0}-s(n-s))
x^{s+2j}\mc{A}_{j,N}(s)f_j +  x^{s + 2j}f_j = x^{s+2N+2}\mc{B}_{j,N}(s)f_j,
\] 
for $j=1,\dots, N$, 
where the term on the right-hand side is in $x^{s+2N+2}{x'}^sC^\infty(\bbar{B}_0\x \bbar{B}_0)$.
Furthermore,
\[\frac{\mc{A}_{j,N}(s)}{\Gamma(s-n/2+j)} \textrm{ and }\frac{\mc{B}_{j,N}(s)}{\Gamma(s-n/2+j)}\textrm{ are holomorphic in }s.\]

Our improved parametrix and error term are given by
\begin{equation}\label{QNEN}
\begin{split}
Q_N(s) &:= \phvert_\delta\phhor R_0(s)\chi_\delta
+\phvert_\delta \sum_{j=1}^{N}x^{s+2j} 
\mc{A}_{j,N}(s)  [\Delta_y,\phhor]M_{j-1}(s)\chi_\delta,\\
E_N(s) &:=  [\Delta_{g_0},\phvert_\delta] \left(\phhor R_0(s)\chi_\delta
+\sum_{j=1}^{N}x^{s+2j}\mc{A}_{j,N}(s) [\Delta_y,\phhor]M_{j-1}(s)\chi_\delta\right) \\
&\qquad + \phvert_\delta [\Delta_y,\phhor]R_{0,N}(s)\chi_\delta + \phvert_\delta\sum_{j=1}^{N}x^{s+2N+2}\mc{B}_{j,N}(s) [\Delta_y,\phhor]M_{j-1}(s)\chi_\delta.
\end{split}
\end{equation}
This construction yields the following:
\begin{prop}\label{parareg}
Let $N\in\nn$ be large, then there exist operators $Q_N(s),E_N(s)$ defined in \eqref{QNEN}
such that
\[ (\Delta_{g_0}-s(n-s))Q_N(s)=\chi_\delta+E_N(s)\] 
and $Q_N(s),E_N(s)$ are meromorphic in $\{\re s>n/2-N\}$ with simple poles at $-k\in-\nn_0$
and with residue an operator of rank $\mc{O}(k^n)$. 
The operator $Q_N(s):x^NL^2(B_0)\to x^{-N}L^2(B_0)$ is bounded  for $\re s>n/2-N$ and $s\notin -\nn_0$, and the Schwartz kernel 
of the error term $E_N(s)$ can be written as
$$
E_N(s; \cdot,\cdot) =  [\Delta_{g_0},\phvert_\delta] x^{s} h_{{\rm cpt}}(s;w,w') {x'}^{s} + \phvert_\delta x^{s+2N+2} h_N(s;w,w') {x'}^{s},
$$
where $h_{{\rm cpt}},h_N$ are smooth functions in  
$\supp(\phhor\nabla \phvert_\delta)\x \supp(\chi_\delta)$ and $\bbar{B}_0\x\bbar{B}_0$, respectively.  On these domains they satisfy quasi-analytic derivative bounds: for $\abs{\alpha}\leq 8N$ and $\dist(s,-\bbN_0)>\vep$
\begin{equation}\label{EN.bounds1}
\abs{\pl_{w}^\alpha h_N(s; w,w')} \le C^{\abs{\alpha}} N^{\abs{\alpha}} e^{C\brak{s}},
\end{equation}
and
\begin{equation}\label{EN.bounds2}
\abs{\pl_{w}^\alpha h_{\rm cpt}(s; w,w')}  \le 
\begin{cases}
e^{C\brak{s}} (C/\sqrt{\delta})^{\abs{\alpha}+n} N^{\abs{\alpha}}, &  \re s \le n/2\\
(C/\sqrt{\delta})^{\abs{\alpha}+2(\re s+N)}N^{\abs{\alpha}}e^{c \abs{\im s}}, &  \re s\geq n/2
\end{cases}
\end{equation}
where $C>0,c>0$ are independent of $s,N,\alpha,\delta$.
\end{prop}  
\begin{proof}
Consider first the poles of $Q_N(s)$.  The poles of the first term, $\phvert_\delta\phhor R_0(s)\chi_\delta$, are accounted for by \eqref{rankRes}.  For the second term, we only need to consider poles coming from $\Gamma(s - n/2 + j)M_{j-1}(s)$,
since $\mc{A}_{j,N}(s)/\Gamma(s-n/2+j)$ is analytic as noted above.  Lemma~\ref{expleft} shows that these poles occur only at $-k\in\nn_0$, and gives
\[
\res_{s=-k} \left(\phvert_\delta \sum_{j=1}^{N}x^{s+2j}\mc{A}_{j,N}(s) [\Delta_y,\phhor]M_{j-1}(s)\chi_\delta \right) 
= \sum_{i=1}^{J(k)} \tilde{u}_{k,i}(s) \otimes v_{k,i},
\]
where the $\tilde{u}_{k,i}(s)$ are given by a sum of analytic differential operators applied to the factors $u^{(\ell)}_{k,i}$ 
appearing in \eqref{Ress=-k}.  The rank of this residue is thus still bounded by $J(k) = \mc{O}(k^n)$.  
The main point here is that we were able to choose the $v_{k,i}$
in \eqref{Ress=-k} independent of $\ell$;  hence the rank estimate is not affected by the sum over $j = 1$ to $N$.
The same reasoning applies to the poles of $E_N(s)$.

The continuity of $Q_N(s) :x^{N}L^2(B_0)\to x^{-N}L^2(B_0)$ comes directly from the boundedness of $\phhor\phvert_\delta R_0(s)\chi_\delta:x^NL^2(B_0)\to x^{-N}L^2(B_0)$  
for $\re s>n/2-N$, which follows easily from the expression \eqref{R0.sigma} 
(see for instance \cite[Prop 3.1, 3.2 and B.1]{Perry:1999}).

Based on \eqref{QNEN} we set 
\[
h_{\rm cpt}(s;w,w') :=  \phhor(y) x^{-s}R_0(s;w,w')x^{-s}\chi_\delta(w')+\sum_{j=1}^{N}x^{2j}
\bigl(\mc{A}_{j,N}  f_j\bigr)(s;w,w') {x'}^{-s},
\]
where $f_j(s;y,w')$ was given by \eqref{Deffj}, and 
\[
h_N(s;\omega,\omega') = [\Delta_y,\phhor] x^{-s-2N-2}R_{0,N}(s;w,w'){x'}^{-s}\chi_\delta(w') 
+ \sum_{j=1}^{N} \bigl(\mc{B}_{j,N} f_j\bigr)(s;w,w'){x'}^{-s}.
\]

To estimate the kernel of $R_0(s)$ we can appeal to the expansion
\eqref{R0.tau} and the uniform estimate, for $\dist(s,n/2 -\bbN) > \vep$,
$$
\abs{\frac{2^{-2j} \Gamma(s+2j)}{\Gamma(s-n/2+1+j) \Gamma(j+1)}} \le e^{C\brak{s}} \brak{j}^{m},
$$
for some $m\in\nn$ independent of $s,j$, which follows directly from Lemma~\ref{beta.lemma}.  Summing over $j$ gives the bound for $\tau\in \cc$ with $|\tau|>1$
\begin{equation}\label{bigsum}
 \Big|\sum_{j=0}^\infty\frac{\Gamma(s+2j)}{\Gamma(s-n/2+1+j) \Gamma(j+1)}
(2\tau)^{-2j}\Big|\leq e^{C\brak{s}}\min(|\tau|-1,1)^{-m+1}.
\end{equation}
Together with the expression \eqref{R0.tau}, this shows that the Schwartz kernel $(xx')^{-s}R_0(s;x,y,x',y')$ is, away from the diagonal, a real analytic function of the variables 
$(x, y,x',y')\in ([0,1)\x \pl\bbar{B}_0)^2$.  
In particular we can obtain estimates for its derivatives 
in terms of its $L^\infty$ bound in a small complex neighbourhood of $([0,1)\x \pl\bbar{B}_0)^2\setminus {\rm diag}$. 

Let $\eta>0$ be small, let $w':=(x',y')\in [0,\delta]\x \pl\bbar{B}_0$, and $B^c(w',\eta)$  be 
\[B^c(w',\eta):=\left\{w=(x,y)\in \cc \x \cc^n;\> \abs{\re(w)-w'}>\eta,\>\abs{\im(w)} < \eta/\sqrt{2} \right\}. \]  
The function $(x,y)\mapsto 1/\tau(x,y,x',y')$ (with $\tau$ defined by \eqref{tau}) admits an analytic extension in 
$B^c(w',\eta)$ and we have the estimate: 
\[ \abs{\tau(w,w')}^{-1} \leq \frac{2x' \abs{x}}{\eta^2/2+2x'\re(x)}\leq \frac{4\delta \abs{x}}
{\eta^2-4\delta \abs{\re(x)}},
\]
for all $w'\in[0,\delta]\x \pl\bbar{B}_0$ and all $w\in B^c(w',\eta)$ with $\abs{\re(x)} < \eta^2/(4\delta)$.
In particular, if we assume that $\delta \le \eta/6$, then we have
\[
\abs{\tau(w,w')}^{-1} \leq 1/2,\quad\text{for }w\in B^c(w',\eta) \cap \{\abs{x} \le \eta/2\},
\] 
uniformly in $w'\in[0,\delta]\x\pl\bbar{B}_0$.

Similarly, $q(w, w') := xx'\tau = x^2+{x'}^2+\abs{y-y'}^2$ admits an analytic extension in $B^c(w',\eta)$ as a function of $w$.
Under the same conditions as above (in particular $\abs{x} \le \eta/2$), we have 
\begin{equation}\label{q.bounds}
\tfrac13 \eta^2 \le \abs{q(w,w')} \le 2 + \mc{O}(\eta).
\end{equation}
This implies
\[\abs{q(w,w')^{-s} }\leq 
C \eta^{-n} e^{C\brak{s}},\quad \text{for }\re s \le n/2,
\]
for $w\in B^c(w',\eta) \cap \{\abs{x} \le \eta/2\}$ as above.
Combining these estimates of $\abs{\tau}^{-1}$ and $\abs{q^{-s}}$ with \eqref{R0.tau} and \eqref{bigsum}, we deduce that for 
$w\in B^c(w',\eta) \cap \{\abs{x} \le \eta/2\}$,
\begin{equation}\label{R0.offdiag}
\abs{(xx')^{-s} R_0(s;w,w')} \le C_\eta e^{C\brak{s}},\quad \text{for }\re s \le n/2.
\end{equation}

For $\re s \ge n/2$ we can improve on \eqref{R0.offdiag} using the expression \eqref{hypergeom} for the resolvent in $\re s > (n-1)/2$.   As above, we will assume that $\delta \le \eta/6$, so that \eqref{q.bounds} holds for
$w\in B^c(w',\eta)$ and $\abs{x} \le \eta/2$.  Furthermore, under the same assumptions, $\sigma(w,w')$ admits an analytic extension such that
\[
\abs{xx'(\sigma-t)} = \abs{\tfrac14 (w - w')^2 - t xx'} \ge \tfrac{1}{24} \eta^2.
\]
Therefore, from \eqref{hypergeom}, we deduce directly that for $w\in B^c(w',\eta) \cap \{\abs{x} \le \eta/2\}$,
\begin{equation}\label{R0.offdiag2}
\abs{(xx')^{-s} R_0(s;x,y,x',y')} \leq (C/\eta^2)^{\re s}e^{c\abs{\im s}},\quad \text{for }\re s \ge n/2.
\end{equation}

The  estimates \eqref{R0.offdiag} and \eqref{R0.offdiag2} are valid  in a complex $w$-neighbourhood with diameter $\eta$;
we thus obtain analytic estimates for derivatives of $(xx')^{-s} R_0(s;x,y,x',y')$ with respect to $w$.
Combining with the quasi-analytic estimates \eqref{QAE} of $\phhor$, and returning to the real variables, this implies that 
for $\abs{w-w'}>\eta$ with $x \le \eta/2$, we have 
\begin{equation}\label{bdR_0}
\abs{\pl_{w}^\alpha \left[x^{-s}\phhor(y)R_0(s;w,w'){x'}^{-s}\right]} \le \begin{cases}
C e^{C\brak{s}} (C/\eta)^{\abs{\alpha}+n} N^{\abs{\alpha}}, &  \re s\leq n/2,\\
(C/\eta)^{\abs{\alpha}+2\re s}N^{\abs{\alpha}}e^{c \abs{\im s}}, &  \re s\geq n/2.
\end{cases}
\end{equation}
for  $\abs{\alpha}\leq 10N$, where $C$ are constants independent of $\eta,\delta,N,s$.

We also note that $f_j(s;y,w'){x'}^{-s}$ (with $f_j$ defined in \eqref{Deffj}) is the $(j-1)$-th term in the Taylor expansion 
of the kernel $[\Delta_y,\phhor(y)] x^{-s}R_0(s;w,w')\chi_\delta(w'){x'}^{-s} 
\in C^\infty(\bbar{B}_0\x\bbar{B}_0)$ in powers of $x^2$ at $x=0$.   By the assumption on the support of 
$\phhor$ and $\chi_\delta$, there is $\eta>0$ independent of $\delta$ such that 
$\abs{w-w'}>\eta$ on the support of the kernel $[\Delta_y,\phhor(y)] x^{-s}R_0(s;w,w')\chi_\delta(w'){x'}^{-s}$. 
Thus there exists $C_0>0$ independent of $\delta,N,s,j$ such that 
$f_j(s;y,w'){x'}^{-s}$ satisfies the hypotheses \eqref{df.quasi} with $A=10$.
The kernel of 
$[\Delta_y,\phhor] x^{-s}R_{0,N}(s)\chi_\delta {x'}^{-s}$ is likewise the remainder in the Taylor expansion at $x=0$, so it satisfies the same quasi-analytic estimates as \eqref{bdR_0}.
The estimates on $h_N$ now follow from Lemma~\ref{f.abc} applied to $f_j(s,y,w')$, together with the corresponding estimates for the $R_{0,N}(s)$ term.  

For the  estimate for $h_{\rm cpt}$, it is the same argument, except 
that to deal with the kernel of $[\Delta_{g_0},\phvert_\delta]\phhor R_0(s)\chi_\delta$, the distance between supports 
satisfies only $\abs{w-w'} \ge \sqrt{\delta} - \delta$.  So at best we can chose something like $\eta=\frac{1}{2} \sqrt{\delta}$. 
\end{proof}

\bigbreak
The construction of Proposition \ref{parareg} would suffice for global bound following the proof in 
Guillop\'e-Zworski \cite{GZ:1995b}, which yields a non-optimal exponent $n+2$ for bounds on resonances. 
To accomplish the dimensional reduction to the exponent $n+1$, as in Cuevas-Vodev \cite{CV:2003}, an additional trick is needed.

Before stating the proposition, let us recall that the weight function $\rho$ of Lemma \ref{constrho} satisfies the bound in the regular neighbourhood (identified with $B_0$)
\begin{equation}\label{brho.bnds}  
x \leq \rho \leq Cx/\delta ,\quad \rho \leq 1, \quad \rho=1 \,\,\textrm{ in }\{x\geq \delta\}.
\end{equation}

\begin{prop}\label{GZ.prop}
For $\delta, \rho, \chi_\delta$ as above and for any $N\in \nn$
there exist meromorphic families of operators with poles at $-\nn_0$ of finite rank, 
\[
\begin{split}
S_N(s) &:\rho^NL^2(B_0)\to \rho^{-N}L^2(B_0),\\
K_N(s) &:\rho^NL^2(B_0)\to \rho^NL^2(B_0), \\
L_N(s) &:\rho^NL^2(B_0)\to \rho^NL^2(B_0), \\
\end{split}
\] 
for $\re s > n/2-N$, such that
\[(\Delta_{g_0}-s(n-s))S_N(s)=\chi_\delta+K_N(s)+L_N(s), 
\]
and the $K_N(s), L_N(s)$ are trace class.  Consider the sectorial region,
\begin{equation}\label{defUN}
U_N:=\bigl\{\abs{\im s}\leq N+4\re s)\bigr\}\cap \bigl\{\dist(s,-\nn_0)>\eps\bigr\}.
\end{equation}
illustrated in Figure~\ref{UNsector.fig}.
For $s \in U_N$, and $\delta>0$ sufficiently small, the singular values of $K_N(s)$ and $L_N(s)$, as operators on $\rho^NL^2(B_0)$, satisfy the following bounds: 
\begin{equation}\label{reg.muK}
\mu_j(K_N(s))\leq e^{-c_\delta N} j^{-2},  \quad j\geq 1,
\end{equation}
and 
\begin{equation}\label{reg.muL}
\mu_j(L_N(s))\leq \begin{cases} C_\delta^N  & \\
e^{-cN} j^{-2}& \text{for } j\ge B_\delta N^n. \end{cases}
\end{equation}
Here all constants $c_\delta,C_\delta,B_\delta,c$ are positive independent of $s, N,k$, and only those indicated depend on $\delta$. 
Moreover, assuming $\delta$ sufficiently small and for $s_N \ge 2N$ we have the estimate
\begin{equation}\label{normLsN}
\norm{L_N(s_N)}_{\rho^NL^2}\leq e^{-c_\delta N}.
\end{equation}
The operators $K_N(s)$ and $L_N(s)$ have possible finite order poles at 
$s = -k$ for $k \in \bbN_0$, and the polar part in the Laurent expansion are
some operators of rank bounded by $\mc{O}(k^{n})$.
\end{prop}

\begin{figure}
\begin{center}
\begin{overpic}{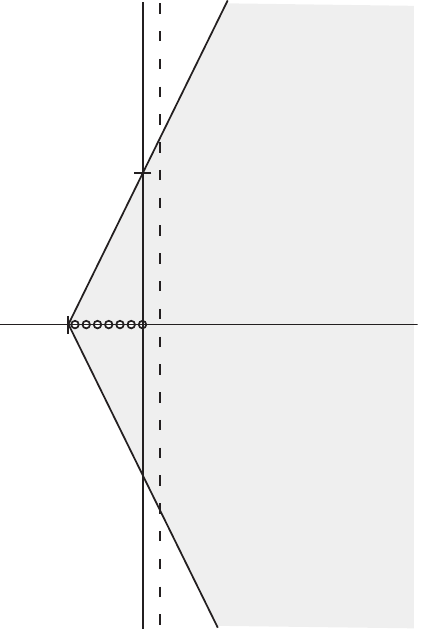}
\put(50,70){$U_N$}
\put(-1,44){$\scriptstyle -N/4$}
\put(16.5,71.5){$\scriptstyle N$}
\put(25.2,44){$\tfrac{n}2$}
\end{overpic}
\end{center}
\caption{The region $U_N$ consisting of a sector centered at $-N/4$ with small disks removed at negative integers.}\label{UNsector.fig}
\end{figure}
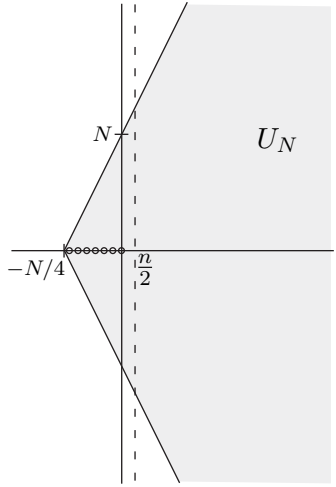

\begin{proof}
The proof of the singular value estimates relies on the dimensional reduction trick introduced by Cuevas-Vodev \cite[Lemma~2.1]{CV:2003}.
 
In first phase of the construction, we apply the parametrix exactly as in Proposition~\ref{parareg}, with the outer cutoff in the product form $\phvert_\delta(x) \phhor(y)$, where $\phvert_\delta \phhor = 1$ on $\supp \chi_\delta$.  Eventually $\phvert_\delta$ will be replaced by a 
step function, in order to accomplish the dimensional reduction.
This produces a parametrix $Q_N(s)$ such that
$$
(\Delta_{g_0} - s(n-s)) Q_N(s) = \chi_\delta + E_N(s),
$$
where $E_N(s)$ has the form
$$
E_N(s; \cdot,\cdot) =  [\Delta_{g_0},\phvert_\delta] x^{s} h_{\rm cpt}(s;\cdot,\cdot) {x'}^{s} + \phvert_\delta x^{s+2N+2} h_N(s;\cdot,\cdot) {x'}^{s},
$$
with compactly supported functions $h_{\rm cpt}$ and $h_N$ that satisfy quasi-analytic derivative estimates.

We cannot simply replace $\phvert_\delta$ by a step function in this expression, because that would
change the order of $E_N(s)$.  To avoid this we use the trick from \cite{CV:2003} of introducing another parametrix further out.
Choose $\tilde\chi \in \cinf_0(\bbar{B}_0)$ independent of $\delta$, such that $\tilde\chi=1$ on the support of $\phhor\phvert_\delta$.  Then apply 
Proposition ~\ref{parareg}
once again to produce a parametrix $\tilde{Q}_N(s)$ satisfying
$$
(\Delta_{g_0} - s(n-s)) \tilde{Q}_N(s) = \tilde\chi + \tilde{E}_N(s).
$$

We replace $Q_N(s)$ by
$$
S_N(s;\cdot,\cdot) = Q_N(s;\cdot,\cdot) 
- \tilde{Q}_N(s) [\Delta_{g_0},\phvert_\delta] x^{s} h_{\rm cpt}(s;\cdot,\cdot) {x'}^{s},
$$
and then exploit the supports of the cutoffs to compute that
\begin{equation}\label{DQN}
(\Delta - s(n-s))S_N(s) = \chi_\delta + K_N(s) + L_N(s),
\end{equation}
where
$$
K_N(s;\cdot,\cdot) = \phvert_\delta x^{s+2N+2} h_N(s;\cdot,\cdot) {x'}^{s},
$$
and 
$$
L_N(s; \cdot,\cdot) = - \tilde{E}_N(s) [\Delta_{g_0},\phvert_\delta] x^{s} h_{\rm cpt}(s;\cdot,\cdot) {x'}^{s}.
$$
The point of this procedure is that $[\Delta_{g_0}, \phvert_\delta]$ is now sandwiched 
between smooth kernels in the $L_N(s)$ term, and so
we can take the distributional limit as $\phvert_\delta$ tends to the step function $H_\delta(x) := 
H(\sqrt{\delta}-x)$ where $H$ is Heaviside function.

After this replacement the first error term can be written as
$$
K_N(s) = H_\delta F_N(s), \textrm{ with }  F_N:=x^{s+2N+2} h_N(s;\cdot,\cdot) {x'}^{s},
$$
and for the claimed estimate it is equivalent to consider the operator $\rho^{-N} K_N(s) \rho^N$ acting on $L^2(\Hyp)$.
To absorb the factor $\rho^{-N}$ on the left, we note that $x/\rho \le 1$ by \eqref{brho.bnds}, 
so that $\rho^{-N} x^{\re s+2N} \le (C\delta)^{\re s/2 + N/2}$ for $x \le 2\sqrt{\delta}$.
On the right, we can use $\rho\le 1$ and $\rho/x \le C/\delta$ to bound $\abs{{\rho'}^N {x'}^s} \le C^{\abs{[\re s]_-}} \delta^{\re s}$ when $x'\leq \delta$, here $\rho'=\rho(w')$.
For $s \in U_N$, the $h_N$ bound \eqref{EN.bounds1} gives 
\[
\abs{\pl_{w}^\alpha h_N(s; w,w')} \le C^{\abs{\alpha}+N + \re s} N^{\abs{\alpha}}.
\]

We can thus derive the estimate,
\[
\abs{\Delta_{w}^{n+1} \rho^{-N} F_N(s) {\rho'}^N} \le 
C^{N+\re s}\delta^{\frac32 \re s+N/2}, \quad\text{for }s \in U_N.
\]
For $\delta$ sufficiently small we will have $C^{\re s} \delta^{\frac32 \re s} \le 1$ for $\re s \ge 0$.  And for $\re s <0$ we note that $\re s \ge -N/4$ for $s\in U_N$, so that $\tfrac32 \re s + N/2 \ge N/8$.  Thus, at worst the bound is $(C\delta)^{N/4}$, so that for $\delta$ sufficiently small we have
\[
\abs{\Delta_{w}^{n+1} \rho^{-N} F_N(s) {\rho'}^N} \le e^{-c_\delta N},
\]
for $s$ in the region $U_N$ defined by \eqref{defUN}.
We can use this to make a comparison to the Dirichlet Laplacian $\Delta_K$ on a compact domain $K$ in the $w=(x, y)$ space.  Using the fact that
\begin{equation}\label{mujK}
\mu_j(\Delta_K^{-(n+1)}) \sim c_K j^{-2},
\end{equation}
we obtain that there is $c_\delta>0$ such that  (identifying operators and kernels)
$$
\mu_j\left(\rho^{-N} F_N(s) \rho^N\right) \le e^{-c_\delta N} j^{-2},
$$
for $s \in U_N$.  The bounds \eqref{reg.muK} follow immediately since $\norm{H_\delta} \le 1$.

For the second error term in \eqref{DQN}, the distributional limit gives the operator (viewing $h_{\rm cpt}$ as an 
operator through its Schwartz kernel)
$$
L_N(s) := \tilde{E}_N(s) [\Delta_{g_0}, H_\delta] x^sh_{\rm cpt}(s){x}^{s}, 
$$
where the commutator is of the form
\begin{equation}\label{comm.limit}
[\Delta_{g_0}, H_\delta] = c_1 \mu_\delta' + c_2 \mu_\delta \del_{x} +
c_3 \mu_\delta.
\end{equation}
where $\mu_\delta$ denotes the Dirac mass at $\sqrt{\delta}$ as a distribution on $\bbR$. 
Let $\Sigma\subset \bbR^n$ be a compact set containing the support of $\phhor$ and $\Delta_\Sigma$ the corresponding Dirichlet Laplacian.  
By (\ref{comm.limit}) we can write
$$
L_N(s) = \sum_{j=1}^3 B_j A_j,	 
$$
where $A_j \rho^N : L^2(B_0,dg_0) \to L^2(\Sigma)$ and $\rho^{-N} B_j: L^2(\Sigma) \to  L^2(B_0,dg_0)$.
We will estimate singular values by rewriting these terms as $B_j \Delta_\Sigma^{-l} \Delta_\Sigma^l A_j$
and exploiting the fact that $\mu_j(\Delta_\Sigma^{-l}) \le Cj^{-2l/n}$ on $L^2(\Sigma)$ since $\Sigma$ is
$n$-dimensional.

The kernel of each $A_j$ is a constant times 
$s^m(\sqrt{\delta})^{s-m} \pl_x^{1-m}h_{\rm cpt}(s; \sqrt{\delta},y, x', y') {x'}^{s}$ for $m=0,1$, 
and  by using \eqref{EN.bounds2}, we get for all $y\in \Sigma$, 
$(x',y')\in\supp(\chi_\delta)$
\[
\abs{\pl_y^\alpha h_{\rm cpt}(s;\delta,y,x',y')} \le 
\begin{cases}
C^N (C/\sqrt{\delta})^{\abs{\alpha}+n} N^{\abs{\alpha}}, & s\in U_N,\>  \re s \le n/2\\
(C/\sqrt{\delta})^{\abs{\alpha}+2(\re s+N)}N^{\abs{\alpha}}, & s\in U_N,\>   \re s\geq n/2
\end{cases}
\]
for $\abs{\alpha}\leq 8N$.  If we apply this to estimate $A_j$, the extra factors of $(\sqrt{\delta})^{s}$ and ${x'}^s$ (for $x' \le \delta$) contribute a factor $\delta^{3\re s/2}$ to the estimate for $\re s \ge 0$.
We thus have
\begin{equation}\label{Ajnorm}
\norm{\Delta_\Sigma^l A_j \rho^N}_{ L^2(B_0) \to L^2(\Sigma)} \le 
\begin{cases}
C_\delta^{N+2l} N^{2l} , & s \in U_N, \>\re s\leq n/2, \\
C^{\re s+N+l}\delta^{\re s/2-l-N} N^{2l}, & s \in U_N,\>\re s\ge n/2.
\end{cases}
\end{equation}
for some constant $C_\delta$ depending on $\delta$ and $l\leq 4N$.  Similarly,
\begin{equation}\label{Bjnorm}
\norm{\rho^{-N} B_j}_{L^2(\Sigma) \to  L^2(B_0)} \le 
\begin{cases}
C^N, & \abs{s-n/2}\leq \gamma N, \>\re s\leq n/2, \\
C^{\re s+N}, & \re s\ge n/2.
\end{cases}
\end{equation}
independent of $\delta$ since $\tilde{E}_N(s)$ did not depend on $\delta$.  The norm estimate of $L_N(s)$ follows by combining the $l=0$ case of \eqref{Ajnorm} with \eqref{Bjnorm}.  
In particular we notice that for $s_N \ge 3N$, we get 
$\norm{L_N(s_N)}_{\rho^NL^2} \leq (C\delta)^{N/2}$, which proves \eqref{normLsN} if $\delta$ is sufficiently small. 

By taking $l = N + n$ in \eqref{Ajnorm} we obtain the singular value estimate, for $s\in U_N$ and assuming $\delta$ is sufficiently small,
\[
\begin{split}
\mu_j\left(L_N(s)\right)
& \le C_\delta^{N} N^{2N}  \mu_j(\Delta_\Sigma^{-N-n}) \\
& \le C_\delta^{N} N^{2N} j^{-2-2N/n}.
\end{split}
\]
To simplify this expression, observe that
$$
j^{-2N/n} N^{2N} \le e^{-bN},\qquad\text{for }j \ge e^{bn/2} N^n.
$$
thus choosing $b$ large enough, depending on $\delta$, we obtain the desired result.

Finally, we note that the estimate on the order of the poles follows directly from the corresponding estimate in Proposition~\ref{parareg}.
\end{proof}

\bigbreak
\section{Parametrix construction for cusp neighbourhoods}\label{par.cusp.sec}

We describe the resolvent of the Laplacian $\Delta$ on a quotient $X_c:=\Gamma^c\backslash\hh^{n+1}$ by an abelian parabolic group of rank $k_0\in[1,n-1]$, fixing $\infty$ in the half-space model and generated by some elements $\gamma_1,\dots,\gamma_{k_0}$.  These act on $\bbH^{n+1} = \bbR^+_x \x \bbR^{n-k_0}_y \x \bbR^{k_0}_z$ by
\[
\gamma_j(x,y,z)= (x, A_j y, z+v_j),
\]
where $A_j \in SO(n-k_0)$ are self-commuting and $v_j \in \bbR^{k_0}$. \\

We will use an additional weight function in $X_c$ which is independent of $\delta$ and given by 
\[ \rho_c:=x/(1+x).\]
Recall from Lemma \ref{constrho} that the $\delta$-dependent weight $\rho$ is a global function on $X$, which satisfies 
 $\rho = 1$ for $x \ge \delta$, $\rho\leq 1$ and 
\begin{equation}\label{brho.cusp}
x \le \rho \le Cx/\delta, \quad x\leq \delta.
\end{equation}

\subsection{Spectral decomposition of $\Delta$ in a cusp.}\label{FB.sec}

In what follows,  $X_c$ is viewed as $\bbR_x^+\x F$ with metric $(dx^2+g_F)/x^2$ where $g_F$ is a flat metric on a flat vector bundle $F=\Gamma^c\backslash \bbR^n$ with base a flat $k_0$-dimensional torus $T=:\bbR^{k_0}/\Lambda$ where $\Lambda$ is the lattice spanned by the $v_j$'s.  Let us use the notation $\bbar{X}_c=[0,\infty)\x F$. 
The Laplacian is 
\begin{equation}\label{DX.def}
\Delta_{X_c} = - (x\del_x)^2 + nx\del_x + x^2 \Delta_F,
\end{equation}
acting on $L^2(X_c, x^{-(n+1)} dx\otimes dv_F)$.

We need recall some details of the Fourier-Bessel decomposition of $L^2(X_c)$ from Guillarmou-Mazzeo 
\cite[\S4]{GM:2012}.
The fibers of $F$ are isometric to Euclidean $\bbR^{n-k_0}$, so a polar decomposition $y=r \omega$
with $r=|y|$ in the fibers gives
\[
\Delta_F = - \del_r^2 - \frac{n-k_0-1}{r} \del_r + \frac{1}{r^2} \Delta_{S^{n-k_0-1}} + \Delta_z.
\]
We can further decompose the $L^2$ space of the unit sphere bundle $SF$ of $F$ as a sum of complex line bundles,
\[
L^2(SF) = \oplus_{m=0}^{\infty} \oplus_{p=1}^{\mu_m} \mc{L}^{(m)}_p.
\]
where $SF$ correspond to the submanifold $\{|y|=1\}$ in $\Gamma^c\backslash \bbR^n$ if $\bbR^n=\bbR^{n-k_0}_y\x \bbR_z^{k_0}$.

A section of $\mc{L}^{(m)}_p$ is identified with a function $f(z,\omega)$ on $\bbR^{k_0} \x S^{n-k_0-1}$, such that 
$f(z,\cdot)$ is a spherical harmonic of degree $m$, with the action of the generators $\gamma_j$ of $\Gamma^c$
given by 
\[
f(z+v_j, \omega) = e^{i\alpha_{mpj}} f(z,\omega),
\]
for some holonomy constants $\alpha_{mpj}$ defined as in the Introduction.  The final step is a Fourier decomposition of the sections of 
$\mc{L}^{(m)}_p$, indexed by $v^* \in \Lambda^*$, the lattice in $\bbR^{k_0}$ dual to $\Lambda$:
this corresponds to decompose in Fourier series in $T$ the $\Lambda$-periodic function $e^{-2\pi i\cjg z,A_{mp}\cjd}f(z,\omega)$
where $2\pi A_{mp}:=\sum_{j=1}^{k_0}\alpha_{mpj}v_j^*$ and $\{v_j^*\}$ is the basis for $\Lambda^*$ dual to $\{v_j\}$.
This decomposition yields an orthonormal basis $\{\phi_I\}_{I \in \mc{I}}$ of $L^2(SF)$ indexed by 
\[
\mc{I} := \{(m,p,v^*) \in \bbN_0 \x \bbN \x \Lambda^*:\> 1\le p \le \mu_m\},
\]
such that if $f\in L^2(F)$ is decomposed as 
\[
f(z,r,\omega)=\sum_{I\in \mc{I}}f_I(r)\phi_I(z,\omega),
\]
then 
\[
\Delta_Ff(z,r,\omega)=\sum_{I\in \mc{I}}(\Delta_If_I)(r)\phi_I(z,\omega)
\] 
and the operator $\Delta_I$ acts on $L^2(\bbR^+, r^{n-k_0-1}\>dr)$  by
\begin{equation}\label{DI.def}
\Delta_I = -\del_r^2 - \frac{n-k-1}{r} \del_r + \frac{m(m+n-k_0-2)}{r^2} + b_I^2,
\end{equation}
with 
\[
b_I := \biggl| \sum_{j=1}^{k_0} \alpha_{mpj}v_j^* + 2\pi v^*\biggr|.
\]

\subsection{Diophantine condition}\label{Dioph.sec}

We decompose the index set $\mc{I}$ according to the values of $b_I$:
\[
\mc{I}_0 := \{I\in \mc{I}:\>b_I=0\},\qquad \mc{I}_> := \{I\in \mc{I}:\>b_I>0\}.
\]
To describe the estimates for the resolvent in the irrational holonomy case, we introduce the function on $\bbR$,
\begin{equation}\label{Lambda.def}
\Lambda_{\Gamma^c}(u) := 2 \brak{u} \log \brak{u} + 
\sup\limits_{I\in \mc{I}_>, \>m \le \abs{u}}  \left[ 2(\abs{u}-m) \log \frac{1}{b_I} - 2m \log m \right] 
\end{equation}

We will say that a cusp $X_c$ satisfies the \emph{Diophantine condition} if for some $c>0$, 
$\gamma\ge 0$, 
\begin{equation}\label{Dioph}
b_I > cm^{-\gamma},\>\text{ for }I \in \mc{I}_>.
\end{equation}
Under this condition a straightforward estimate gives
\[
\sup\limits_{I\in \mc{I}_>, \>m \le \abs{u}}  \left[ 2(\abs{u}-m) \log \frac{1}{b_I} - 2m \log m \right] 
\le 2\gamma \abs{u} \log \abs{u},
\]
so that $\Lambda_{\Gamma^c}(u)$ has the minimal growth rate, 
\[
\Lambda_{\Gamma^c}(u) = \mc{O}(\brak{u} \log \brak{u}).
\]

To illustrate the behavior of $\Lambda_{\Gamma^c}$, let us consider the simplest non-trivial example, a rank one cusp in $\bbH^{4}$.
The group $\Gamma^c$ is cyclic with generator,
\[
\gamma(x,y,z) = (x, R_\theta y,z+\ell),
\]
where $R_\theta \in SO(2)$ denotes the rotation by angle $\theta$.  
For rank one it is natural to let the index $m$ range over 
$\bbZ$, so that the multiplicities are all $\mu_m = 1$ and there is no need for the index $p$.  The dual lattice $\Lambda^* = \bbZ/\ell$, so the modes are indexed by $I = (m,j) \in \bbZ\x \bbZ$, and we have
\[
b_I = \frac{2\pi}{\ell}  \abs{\frac{m\theta}{2\pi} + j}.
\]
If $\theta/(2\pi)$ is rational, then $b_I$ is bounded below by a constant for $I \in \mc{I}_>$ so the Diophantine condition is trivially satisfied.
And if $\theta/2\pi$ is an algebraic number, then Roth's theorem on Diophantine approximation \cite{Roth:1955} implies that \eqref{Dioph} holds for any $\gamma>1$.  

On the other hand, if $\theta/(2\pi)$ is transcendental then $\Lambda_{\Gamma^c}(u)$ could grow more rapidly.  For example, define $a_k$ recursively by 
\[a_1 = 2, \quad a_{l+1} = 2^{a_l^q},\]
for some $q \in \bbN$.  Then set
$\theta = 2\pi \sum_{l=1}^\infty (1/a_l)$.  With $m = a_k$, and $j = - \sum_{l=1}^k (a_k/a_l)$,
we find that
$$
b_I = \frac{2\pi}{\ell} \sum_{j=k+1}^\infty \frac{a_k}{a_j}
\approx \frac{2\pi m}{\ell} \> 2^{-m^q}.
$$
This would give $\Lambda_{\Gamma^c}(u) \asymp \abs{u}^{q+1}$.  It is clear that by modifying this construction we could produce angles for which $\Lambda_{\Gamma^c}(u)$ would grow arbitrarily rapidly.

\subsection{Resolvent estimates}
The meromorphic continuation of $R_{X_c}(s)$ was established in \cite[Prop~5.1]{GM:2012}.  Here we follow that proof but keep track of the $s$-dependence in the estimates.  For the $L^2$ estimates we use a boundary defining function $\rho$ which, just as in \S\ref{par.reg.sec}, will depend on the small parameter $\delta$.
 
\begin{prop}\label{RXc.estimate}
For any $\psi\in C_0^\infty(X_c)$ and any $N>0$, the truncated resolvent $\psi R_{X_c}(s)\psi$ admits a meromorphic extension from $\{\re s>\ndemi\}$ to $\{\re s>n/2-N\}$ as a bounded operator mapping 
$\rho_c^N L^2(X_c)$ to $\rho_c^{-N}L^2(X_c)$, where $\rho_c= x/(x+1)$.   The
poles are contained in $k_0/2-\nn_0$ and each $k_0/2-k$ with $k\in\nn_0$ has rank 
of order $\mc{O}(k^{n-k_0})$.
Moreover, for $\eps>0$, one has the bound in $\{s\in \cc; \re s>n/2-N; d(s,k_0/2-\nn_0)>\eps\}$
$$
\norm{\psi R_{X_c}(s)\psi}_{\rho_c^NL^2\to \rho_c^{-N}L^2} \le  
\begin{cases}
e^{C\brak{s} + \Lambda_{\Gamma^c}([\re s]_-)} &,   \textrm{ if }\re s<n/2+1,\\
C  & ,  \textrm{ if } \re s\geq n/2+1,
\end{cases}
$$
where $C$ is independent of $s,N$, and the 
quantity $\Lambda_{\Gamma^c}$ was defined by \eqref{Lambda.def}.
\end{prop}

\begin{proof} The bound in the physical half-plane $\re s>n/2+1$ just follows from the $L^2\to L^2$ bound 
obtained through the spectral theorem. Let us then consider that $\re s\leq n/2+1$.
For the proof we conjugate $\Delta_{X_c}$ by $x^{n/2}$, so that instead of \eqref{DX.def} we consider the operator
\begin{equation}\label{DX.conj}
\Delta_{X_c} = -(x\del_x)^2 + x^2\Delta_F + \frac{n^2}4,
\end{equation}
acting on $L^2(X_c, \frac{dx}{x}\otimes dv_F)$.
From the spectral resolution of $\Delta_F=\oplus_{I\in \mc{I}}\Delta_I$, the resolvent $R_{X_c}(s)$ 
is also a direct sum 
\[(R_{X_c}(s)f)(x,r,\omega,z)=\sum_{I\in\mc{I}} (R_I(s)f_I)(x,r)\phi_I(z,\omega),\]
where $f(x,z,r\omega)=\sum_{I\in \mc{I}}f_I(x,r)\phi_I(z,\omega)$, with 
$f_I\in L^2(\bbR^+,\frac{dx}{x}; L^2(\bbR^+,r^{n-k_0-1}dr))$.

\medbreak\emph{Case I}: $I \in \mc{I}_0$.  If $b_I=0$ then the corresponding basis element $\phi_I(z,\omega)$ can just be written $\phi_I(\omega)$, independent of the $z$ variable.  This observation yields an isometry,
\begin{equation}\label{isom} 
\biggl\{f\in L^2(X_c, \tfrac{dx}{x}\otimes dv_F):\> f=\sum_{I\in\mc{I}_0}f_I\phi_I \biggr\}
\to L^2(\hh^{n-k_0+1}),
\end{equation}
given by
\[
f\mapsto  x^{\frac{n-k_0}{2}}f.
\]
As explained in the proof of \cite[Prop.~5.1]{GM:2012}, this identification allows us to realize the resolvent component 
$R_I(s)$, for $\re (s)>n/2$, by 
\begin{equation}\label{RIf}
(R_I(s)f_I)(x,r)\phi_I(\omega) = x^{-\frac{n-k_0}{2}}(R_{\hh^{n-k_0+1}}(s-\tfrac{k_0}{2})(x^{\frac{n-k_0}{2}}f_I\phi_I))(x,r,\omega)
\end{equation}
where $R_{\hh^{n-k_0+1}}(\zeta)=(\Delta_{\hh^{n-k_0+1}}-\zeta(n-k_0-\zeta))^{-1}$ is the resolvent of the Laplacian on the lower dimensional hyperbolic space $\hh^{n-k_0+1}$.  This works because $\Delta_{\hh^{n-k_0+1}}$ preserves the decomposition into spherical harmonics coming from the polar decomposition 
$\hh^{n-k_0+1} = \bbR^+ \x (\bbR^+ \x S^{n-k_0-1})$, so that $R_{\hh^{n-k_0+1}}(\zeta)(f_I\phi_I)$ is still a multiple of $\phi_I$.

The meromorphic extension properties of $R_{\hh^{n-k_0+1}}(\zeta)$ are of course clear from \eqref{R0.sigma}, so that 
for $I\in \mc{I}_0$, $\psi R_I(s)\psi$ has a meromorphic extension to $\re s > n/2-N$ as an operator $\rho_c^{N}L^2 \to \rho_c^{-N} L^2$.
The standard estimate,
\[\norm{\psi R_{\hh^{n-k_0+1}}(\zeta)\psi}_{\rho_c^NL^2\to \rho_c^{-N}L^2}=\mc{O}( e^{C\abs{\re(\zeta)}}),\] 
holds when $\re(\zeta)>(n-k_0)/2-N$, and implies 
\[\norm{\psi R_I(s)\psi}_{\rho_c^NL^2\to \rho_c^{-N}L^2}=\mc{O}(e^{C\abs{\re s}}), \quad  \textrm{ when }\re s>n/2-N.\]
For $n-k_0$ odd, $R_I(s)$ has poles at $k_0/2-\nn_0$, with finite rank residue. 
From \eqref{rankRes} we  obtain the upper bound on the ranks for $n-k_0$ odd,
\[ \rank\,\res_{s=\frac{k_0}{2}-k}(\oplus_{I\in\mc{I}_0} R_I(s))\leq 
 \rank\,\res_{\zeta=-k}(R_{\hh^{n-k_0+1}}(\zeta))=\mc{O}(k^{n-k_0}).\]
(For $n-k_0$ even, $\oplus_{I\in\mc{I}_0} R_I(s)$ has no poles.)

\bigbreak\emph{Case II}: $I\in \mc{I}_>$.  
We follow the proof of \cite[Prop.~5.1]{GM:2012}, keeping 
track of the $s$ dependence of the constants.  
The starting point is the representation for the resolvent component $R_I(s)$, based on a standard ODE analysis of \eqref{DX.conj}, as
\begin{equation}\label{Rc.bessel}
R_{I}(s)f_I(x,\cdot) = \int_0^{\infty} F_{s,x,x'}\left(\sqrt{\Delta_I} \right) f_I(x',\cdot)\frac{dx'}{x'},
\end{equation}
where $F$ is defined in terms of Bessel functions,
\[
F_{s,x,x'}(\tau) := K_\lambda(x\tau) I_\lambda(x'\tau) H(x-x') + I_\lambda(x\tau) K_\lambda(x'\tau) H(x'-x),
\]
with $\lambda := s-n/2$.  
The same ODE analysis applied to \eqref{DI.def} yields the functional calculus: for a bounded function $G$,
\begin{equation}\label{DI.funct}
G(\Delta_I) = \int_0^\infty G(t^2 + b_I^2)\>d\Pi_I(t),
\end{equation}
where $d\Pi_I$ is the spectral measure of $\Delta_I-b_I^2$, the Schwartz kernel of which is
\begin{equation}\label{dPi_I}
d\Pi_I(t;r,r') := \frac{2}{\pi i} (rr')^{-\frac{n-k_0-2}2} J_{\frac{n-k_0-2}2+m}(rt) J_{\frac{n-k_0-2}2+m}(r't) \>t\>dt.
\end{equation}

Fix $\epsilon_0 >0$ and choose $m_0$
such that $m \le m_0$ implies $b_I \ge \epsilon_0$ for $I \in \calJ_>$.  Set 
$\mc{I}_{m_0} := \{I \in \mc{I}_>:\> m \le m_0\}$.  The estimate for $I \in \mc{I}_{m_0}$
follows from (\ref{Fs.est}), \eqref{Rc.bessel}, and \eqref{DI.funct}:
$$
\norm{\rho_c^{N} \psi R_I(s) \psi \rho_c^{N}}_{\mc{L}(L^2)} \le C e^{c\abs{\lambda}} \max\left(\abs{\re\lambda}^{-2\re\lambda}, 1\right),
$$
for $\re \lambda > -N$.

For $I \notin \mc{I}_{m_0}$ we first derive the high-frequency estimate,
\begin{equation}\label{hf.RIbound}
\norm{\rho_c^{N} \psi  \mathbbm{1}_{(1,\infty)}(\sqrt{\Delta_I}) R_I(s) \psi \rho_c^{N}}_{\mc{L}(L^2)} 
\le Ce^{c\abs{\lambda}} \max\left(\abs{\re\lambda}^{-2\re\lambda}, 1\right),
\end{equation}
from (\ref{Fs.est}).  For the low frequencies we have $|b_I| \leq 1$.
The expression \eqref{dPi_I} together with the classical bound (see \cite[Chap 9]{AS}),
\[ \abs{J_{\alpha}(rt)} \leq \frac{(r/2)^{\alpha}t^{\alpha}}{\Gamma(\alpha+1)}, \quad \textrm{ for } t<1,\,  \,  \alpha>0,\]
gives pointwise estimates for $d\Pi_I(t;r,r')$.  Therefore, using \eqref{Fs.est} we get,
for $\re \lambda > - N$,
\begin{equation}\label{smallt}
\begin{split}
&\abs{\psi(r) \psi(r') (\rho_c\rho_c')^{N}  \mathbbm{1}_{(0,1)}(\sqrt{\Delta_I}) F_{s,x,x'}(\sqrt{\Delta_I})(r,r')} \\
&\qquad = \abs{\int_0^{\sqrt{1-b_I^2}} \psi(r) \psi(r') (\rho_c\rho_c')^{N} F_{s,x,x'}\left(\sqrt{t^2+b_I^2} \right) \>d\Pi_I(t,r,r')} \\
&\qquad \le C e^{c\abs{\lambda}} \max\left(\abs{\re\lambda}^{-2\re\lambda}, 1\right) \frac{e^{cm}}{\Gamma(m+(n-k)/2)^2} \\
&\hskip1in\times 
\int_0^{\sqrt{1-b_I^2}}  \max\left(\bigl(t^2+b_I^2\bigr)^{\re \lambda},1\right) t^{2m+n-k_0-2}\>dt.
\end{split}
\end{equation}
The final integral is $\mc{O}(1)$ for $2\re \lambda + 2m +n - k_0-1 > 0$.   For $2\re \lambda + 2m +n - k_0-1 < 0$ it is  easily estimated by
$$
\int_0^{\sqrt{1-b_I^2}}(t^2+b_I^2)^{\re \lambda} t^{2m+n-k_0-2}\>dt \le 
C b_I^{2\re \lambda + 2m+n-k_0-1}.
$$
The term \eqref{smallt} is thus bounded by 
$$
 C e^{c \abs{s}} e^{cm} m^{-2m} \begin{cases}1 & \re s > n/2,\\
\abs{\re s}^{-2\re s} & -m \le \re s \le 0, \\
 \abs{\re s}^{-2\re s} b_I^{2\re s + 2m} &  \re s \le -m, \end{cases}.
$$
We conclude that for any $I \notin \mc{I}_{m_0}$,
\begin{equation}\label{lf.RIbound}
\norm{\rho_c^{N} \psi  \mathbbm{1}_{(0,1)}(\sqrt{\Delta_I}) R_I(s) \psi \rho_c^{N}}_{\mc{L}(L^2)} 
\le C e^{c\abs{s} + \Lambda_{\Gamma^c}([\re s]_-)},
\end{equation}
with $C$ independent of $I$.  This final case completes the proof.
\end{proof}

\bigbreak
We will need another lemma, which is based on \cite[Prop 5.3]{GM:2012} and provides 
structure and estimates on derivatives of $R_{X_c}(s;\omega,\omega')$ in some compact sets
of $(\bbar{X}_c\x \bbar{X}_c)\setminus {\rm diag}$.  Recall that $U_N$ is the sectorial region centered at $s = -N/4$, as defined in
\eqref{defUN}.
\begin{lemma}\label{QAEforRXc}
Let $N\in\nn$ be large, and let $\psi_1,\psi_2 \in C_0^\infty(\bbar{X}_c)$ independent of $\delta$, with disjoint supports, 
satisfying quasi-analytic estimates of the form \eqref{QAE} but independent of $\delta$. 
Then the Schwartz kernel $F(s;w,w')$ of $\psi_1 R_{X_c}(s)\psi_2$ lies in $(xx')^{s}C^\infty_0(\bbar{X}_c\x \bbar{X}_c)$ and 
the following estimates hold for  $s\in U_N$ and ${\rm dist}(s,k_0/2-\nn_0)>\eps$, 
\begin{equation}\label{F3d.bd}
\abs{\del_w^\alpha (x^{-s}F(s;w,w'))\rho(w')^N} \le 
C^{\abs{\alpha}+N} N^{\abs{\alpha}}  e^{\Lambda_{\Gamma^c}([\re s]_-)}e^{C \brak{s}} \delta^{-[\re s]_-},
\end{equation}
when $x \leq 3\delta$, and 
\begin{equation}\label{Fd.bd}
\abs{\del_w^\alpha F(s;w,w')\rho(w')^N} \le 
C_\delta^{\abs{\alpha}} N^{\abs{\alpha}}  e^{\Lambda_{\Gamma^c}([\re s]_-)} \delta^{-[\re s]_-}e^{C(|\im s|+N)-c\re s},
\end{equation}
when $x\ge \delta$,  where in both cases for $\abs{\alpha}\leq 8N$. The constants $C>0, c>0, C_\delta>0$ depend on $\delta$ only as indicated;  all are independent of $s$ and $N$ but do depend on $\eps$.
\end{lemma}
\begin{proof} We use the method of Prop 5.3 in Guillarmou-Mazzeo \cite{GM:2012} to reduce this to a combination of 
pointwise estimates of the error terms from the regular parametrix construction and the operator norm estimates for
$R_{X_c}(s)$.  

First we take a small parameter $\eta>0$ which is independent of $\delta$, such that $0<\delta\ll \eta\ll 1$.  At the end of the proof we will fix 
$\eta$ more precisely. We cover each $\supp \psi_i$, $i=1,2$, by some open sets  $\mc{U}^{(i)}_j \subset X_c$ for $j \in J^{(i)}$, which are either boundary neighbourhoods isometric to the half-ball $B_0\subset \hh^{n+1}$ or interior neighbourhoods isometric to 
a full geodesic ball of radius $r_0$ of $\hh^{n+1}$ for some small enough $r_0>0$.  We subdivide the index set accordingly as $J^{(i)}_{\rm bd} \cup J^{(i)}_{\rm int}$.   This can be done simply by using 
a covering of a fundamental domain of $\Gamma^c$ in $\hh^{n+1}$, so that in this way the function $x$ in $X_c$ can be chosen to be the same as that of the chart $B_0$ for each boundary neighbourhood.  
We assume that the $\mc{U}^{(1)}_j$ neighbourhoods are all disjoint from the $\mc{U}^{(2)}_j$.

For these sets of charts $\{\mc{U}^{(i)}_j\}$ covering the supports of $\psi_i$,
we introduce cutoffs $\chi^{(i)}_j,\hat{\chi}^{(i)}_j \in C_0^\infty(\mc{U}^{(i)}_j)$ such that 
$\hat{\chi}^{(i)}_j=1$ on the support of $\chi^{(i)}_j$ and so that
\[
\psi_i := \sum_{j\in J^{(i)}}\chi^{(i)}_j.
\]
We assume that for $j\in J^{(i)}_{\rm bd}$ each $\chi^{(i)}_j$ is supported in $\{x <\eta\}$ and $\hat{\chi}^{(i)}_j$ in 
$\{x< {2\sqrt{\eta}}\}$ just as in Proposition \ref{parareg} but with $\eta$ replacing $\delta$, and 
that for $j\in J^{(i)}_{\rm int}$, $\chi^{(i)}_j$ is supported in $\{x >\eta/2\}$ and $\hat{\chi}^{(i)}_j$ in $\{x> {\eta}/4\}$.  
We also assume that all cutoffs satisfy the quasi-analytic estimates of the form \eqref{QAE} (with constants depending on $\eta$ instead of $\delta$) in the coordinates of $\hh^{n+1}$ given by the charts. 

We start by performing a standard parametrix construction with respect to the cutoff $\psi_2$.  
For $j \in J^{(2)}_{\rm int}$, let $R^j_0(s)$ be the resolvent on $\hh^{n+1}$ pulled back to $\mc{U}^{(2)}_j$.  The interior parametrix is 
\[
Q^{(2)}_{N,\rm int}(s) := \sum_{j \in J^{(2)}_{\rm int}} \hat{\chi}^{(2)}_jR^j_0(s)\chi^{(2)}_j,
\]
which satisfies
\[
(\Delta_{X_c}-s(n-s)) Q^{(2)}_{N,\rm int}(s) =  \sum_{j \in J^{(2)}_{\rm int}} \chi^{(2)}_j 
+ E^{(2)}_{N,\rm int}(s),
\]
with
\[
E^{(2)}_{N,\rm int}(s) = \sum_{j \in J^{(2)}_{\rm int}} \left[\Delta_{X_c},\hat{\chi}^{(2)}_j\right] R_0^j(s)\chi^{(2)}_j.
\]

Let $d_0$ denote the minimum hyperbolic distance between the supports of $\nabla\hat{\chi}^{(2)}_j$ and 
$\chi^{(2)}_j$, for $j \in J^{(2)}_{\rm int}$.  We can derive estimates for $E^{(2)}_{N,\rm int}(s)$ using the representations \eqref{hypergeom} with $\sigma \ge \cosh^2 (d_0/2)$ and \eqref{R0.tau} with 
$\tau \ge \cosh d_0$.  In combination with the quasi-analytic estimates of $\hat{\chi}^{(2)}_j$, this gives for $s\in U_N$ and  ${\rm dist}(s,-\nn_0)>\eps$,
\begin{equation}\label{QAE.ENint}
\abs{\rho_c(w)^{-N}E^{(2)}_{N,\rm int}(s,w,w')\rho(w')^{N}} \le 
e^{-c_\eta \re s +C_\eta (\abs{\im s}+N)},
\end{equation}
where $c_\eta>0$ and $C_\eta>0$ depend on $d_0$ and $\eta$.

For the boundary neighbourhoods $\mc{U}^{(2)}_j$ with $j\in J^{(i)}_{\rm bd}$, we construct parametrices as in Proposition \ref{parareg}, with $\chi^{(i)}_j$ playing the role of $\chi_\delta$ and replacing the parameter $\delta$ by $\eta$.  
Summing these parametrices for $j\in J^{(i)}_{\rm bd}$ gives $Q^{(2)}_{N, \rm bd}(s)$ satisfying
\[
(\Delta_{X_c}-s(n-s))Q^{(2)}_{N, \rm bd}(s)=\sum_{j \in J^{(2)}_{\rm bd}} \chi^{(2)}_j +E^{(2)}_{N, \rm bd}(s).
\]
From \eqref{EN.bounds1} and \eqref{EN.bounds2} with $\delta$ replaced by $\eta$, the form of $E_N(s)$ in Proposition \ref{parareg},  and the quasi-analytic estimates of $\hat{\chi}^{(2)}_j$, we derive the estimate
for $s\in U_N, {\rm dist}(s,n/2-\nn/2)>\eps$
\begin{equation}\label{QAE.ENbd2}
\abs{\rho_c(w)^{-N}E^{(2)}_{N,\rm bd}(s;w,w')\rho(w')^N}\leq 
 \eta^{-{\rm Re}(s)/2}e^{C\brak{s}+C_\eta N}\delta^{-[\re s]_-},
\end{equation}
for $\abs{\alpha}\leq 8N$, where $C>0$ does not depend on $\eta$ but $C_\eta$ does;  note that 
we have used $|{x'}^s\rho(w')^N|\leq (C\delta)^{-[\re s]_-}$ for $x'\leq \eta$. 

Combining the interior and boundary parametrices we conclude that there exist meromorphic operators $Q^{(2)}_N(s):\rho_c^NL^2\to \rho_c^{-N}L^2$, such that 
\begin{equation}\label{rightpara}
(\Delta_{X_c}-s(n-s))Q^{(2)}_N(s)=\psi_2+E^{(2)}_N(s),
\end{equation}
with 
\[
E^{(2)}_N(s) := E^{(2)}_{N,\rm int}(s) + E^{(2)}_{N,\rm bd}(s).
\]
The poles of $Q^{(2)}_N(s)$ and $E^{(2)}_N(s)$ are contained in $-\nn_0$.
 
In the same way, but exchanging the positions of $\hat{\chi}^{(1)}_j$ and $\chi^{(1)}_j$ and solving away boundary terms on the right instead of the left, we can construct $Q^{(1)}_N(s):x^NL^2\to x^{-N}L^2$ so that 
\begin{equation}\label{leftpara}
Q^{(1)}_N(s)(\Delta_{X_c}-s(n-s))=\psi_1+E^{(1)}_N(s).
\end{equation}
Here
\[
E^{(1)}_N(s) = E^{(1)}_{N,\rm int}(s) + E^{(1)}_{N,\rm bd}(s),
\]
with $E^{(1)}_{N,\rm int}(s) \in C_0^\infty(X_c\x X_c)$ and 
$E^{(1)}_{N,\rm bd}(s;\cdot,\cdot) \in x^s{x'}^{s+2N+2}C_0^\infty(\bbar{X}_c\x\bbar{X}_c)$.  All have poles contained in $-\nn_0$. Using the quasi analytic bounds on the cutoffs, the estimates \eqref{QAE.ENint} apply to $E^{(1)}_{N,\rm int}(s)$ also, but with additional derivative bounds,
\begin{equation}\label{QAE.ENint2}
\abs{\pl_{\omega}^\alpha E^{(1)}_{N,\rm int}(s,w,w')\rho_c(w')^{-N}} \le 
C_\eta^{|\alpha|}N^{|\alpha|}e^{-c_\eta \re s +C_\eta( \abs{\im s}+N)},
\end{equation}
for $\abs{\alpha}\leq 8N$, $s\in U_N$ with ${\rm dist}(s,n/2-\nn/2)>\eps$, 
where $c_\eta>0,C_\eta>0$ depend on $d_0$ and on $\eta$.
The same method as in the proof of Lemma \ref{parareg} 
yields estimates of the form of \eqref{QAE.ENbd2}, but with additional derivatives bounds: 
\begin{equation}\label{QAE.ENbd1}
\abs{\pl_w^\alpha (x^{-s}E^{(1)}_{N,\rm bd}(s;w,w'))\rho_c(w')^{-N}}\leq 
C_\eta^{\abs{\alpha}+N} N^{\abs{\alpha}} \eta^{-{\rm Re}(s)/2}e^{C\brak{s}},
\end{equation}
for $\abs{\alpha}\leq 8N$, where $C$ does not depend on $\eta$. We will also need an estimate
 for $|\pl_w^\alpha (E^{(1)}_{N,\rm bd}{\rho_c'}^{-N})|$ in the region $x\in [\delta,\eta]$, and this follows from
\eqref{QAE.ENbd1}  by noticing that, by analyticity of $x$, one has for $|\alpha|\leq 10N$ and $x\in[\delta,\eta]$
\begin{equation}\label{derivex^s}
|\pl_{w}^\alpha (x^{s})|\leq e^{C\cjg s\cjd}C_\delta^{|\alpha|}N^{|\alpha|} x^{\re s}.
\end{equation}
Thus for $x\in[\delta,\eta]$,
\begin{equation}\label{QAE.ENbd11}
\abs{\pl_w^\alpha (E^{(1)}_{N,\rm bd}(s;w,w'))\rho_c(w')^{-N}}\leq 
C_\delta^{\abs{\alpha}+N} N^{\abs{\alpha}} \left(\frac{x}{\sqrt{\eta}}\right)^{\re s}e^{C\brak{s}}.
\end{equation}
Because of the disjointness of the supports, applying $R_{X_c}(s)\psi_2$ on the right in \eqref{leftpara} gives
\[
0 = Q^{(1)}_N(s)\psi_2  = \psi_1 R_{X_c}(s)\psi_2 + E^{(1)}_N(s)R_{X_c}(s) \psi_2,
\]
and then \eqref{rightpara} implies
\[
\psi_1 R_{X_c}(s)\psi_2 = E^{(1)}_N(s)R_{X_c}(s)E^{(2)}_N(s).
\]
We now combine the estimates \eqref{QAE.ENint}, \eqref{QAE.ENbd2}, \eqref{QAE.ENbd1}, in combination with the $L^2$ estimates on $R_{X_c}(s)$ from Proposition \ref{RXc.estimate}: for $s\in U_N$, ${\rm dist}(s,n/2-\nn/2)>\eps$, and  $x\leq 3\delta \ll \eta$,
this yields
\[ |\pl^\alpha_w( x^{-s}F(s;w,w'))\rho(w')^N| \leq C_\eta^{\abs{\alpha}+N} 
N^{\abs{\alpha}} e^{C_\eta\brak{s}} e^{\Lambda_{\Gamma^c}([\re s]_-)}\delta^{-[\re s]_-}.
\]
For $\eta$ fixed independent of $\delta$ gives \eqref{F3d.bd} away from the set $n/2-\nn/2$. 

We next consider the region $x\geq \delta$.  We simply gather the estimates 
\eqref{QAE.ENint}, \eqref{QAE.ENint2},Ê\eqref{QAE.ENbd2}, \eqref{QAE.ENbd11}  together with  Proposition \ref{RXc.estimate}, and fix $\eta$ small independent of $\delta$.  This gives \eqref{Fd.bd} away from the set $n/2-\nn/2$. 

Finally, to obtain the estimate in the epsilon neighbourhood of $(n/2-\nn/2)\setminus 
(k_0/2-\nn_0)$, it suffices to use the maximum principle since we know the operator is analytic there, by Proposition \ref{RXc.estimate}.
\end{proof}

\bigbreak
We now recall Lemma 5.2 in \cite{GM:2012} and add the estimate on the rank of the poles.  (This can be compared to the regular neighbourhood version given in Lemma \ref{expleft}.) 
\begin{lemma}\label{Mellcusp}
Let $A>0$. There exist operators $M_\ell(s):C_0^\infty(X_c)\to C^\infty(F)$, $\ell\in\nn_0$, and 
$R_{X_c,N}(s): C_0^\infty(X_c)\to x^{s+2N}L^\infty(X_c)$, $N\in\nn$,
such that for all $\chi\in C_0^\infty([0,A)\x \bbR^n)$ and all $\varphi\in C_0^\infty(X_c)$, 
\[ (\chi R_{X_c}(s)\varphi)(x,y,z) =\chi(x,y,z) \sum_{\ell=0}^{N-1}x^{s+2\ell}(M_\ell(s)\varphi)(y,z)+
(\chi R_{X_c,N}(s)\varphi)(x,y,z).\]
In addition, $\Gamma(s-\ndemi+\ell+1)M_\ell(s)$ is meromorphic in $s\in\cc$ with at most simple poles
at $k_0/2-k\in k_0/2 -\nn_0\cap\{\re s>n/2-N\}$, and with residue of the form 
\begin{equation}\label{Ress=k0-k}
\res_{s=\frac{k_0}{2}-k}(\Gamma(s-\tfrac{n}{2}+\ell+1)M_\ell(s))=\sum_{i=1}^{J(k)}u^{(\ell)}_{k,i}\otimes v_{k,i},
\end{equation}
for some $u^{(\ell)}_{k,i} \in C^\infty(F)$, $v_{k,i}\in C^\infty(X_c)$, with $J(k)=\mc{O}(k^{n-k_0})$ for $k$ large. 
The operator $R_{X_c,N}(s)$ is meromorphic $\{\re s>n/2-N\}$ with simple poles at each $k_0/2-k\in k_0/2-\nn_0$ and 
the residue has rank $\mc{O}(k^{n-k_0})$.
\end{lemma}
\begin{proof}
The existence of $M_\ell(s)$ is proved in Lemma 5.2 of \cite{GM:2012}, these are the operators 
\[ M_{\ell}(s)\varphi= \frac{2^{-2\ell}}{\ell !\Gamma(s-\ndemi+\ell+1)}\int_{0}^\infty \Delta_F^\ell 
K_{s-\ndemi}\left(x'\sqrt{\Delta_F}\right){x'}^{-\ndemi}\varphi(x',\cdot)\frac{dx'}{x'}.
\]

To analyze these operators, we use the Fourier-Bessel decomposition of $\Delta_F$ introduced in \S\ref{FB.sec}.
Decomposing $\varphi=\sum_{I\in \mc{I}}\varphi_I\phi_I$, where 
$\phi_I$ is the orthonormal basis of $L^2(SF)$, $M_\ell(s)\varphi$ decomposes as a sum
$M_\ell(s)\varphi=\sum_{I\in\mc{I}} (M_{\ell,I}(s)\varphi_I) \phi_I$ for 
\begin{equation}\label{MellI} 
M_{\ell,I}(s)\varphi_I= \frac{2^{-2\ell}}{\ell !\Gamma(s-\ndemi+\ell+1)}\int_{0}^\infty \Delta_I^\ell K_{s-\ndemi}(x'\sqrt{\Delta_I}){x'}^{-\ndemi}\varphi_I(x',\cdot)\frac{dx'}{x'}.
\end{equation}
For each $I\in \mc{I}_>$, this is holomorphic in $s$ since $\Delta_I\geq b_I^2>0$, $x'$ is restricted to a compact interval of $(0,\infty)$ by the support of $\varphi\in C_0^\infty(X_c)$, and $K_{s-\ndemi}(x'\sqrt{\Delta_I})$
is a holomorphic family in $s\in \cc$ of bounded operators on $L^2(\bbR^+,r^{n-k_0-1}dr)$. For $\re s>n/2-N$, the bound in terms of 
$I\in \mc{I}_>$ is uniform in $I$ and depends only on $N$, by the same argument as in the proof of 
Proposition \ref{RXc.estimate}. Therefore 
$\Gamma(s-\ndemi+\ell+1)\sum_{I\in\mc{I}_>}M_{\ell,I}(s)\varphi$ is a holomorphic family of operators on 
$L^2(F)$ for $\re s>n/2-N$.

For the terms with $I\in \mc{I}_0$, we can prove the extension of $M_{\ell,I}(s)$ in $s$ with poles appearing at 
$k_0/2-\nn_0$ from the expression \eqref{MellI}, but it is in fact simpler to use the expression 
\eqref{RIf} to write the restriction of $R_{X_c}(s)$ to the $I\in \mc{I}_0$ components 
in terms of $R_{\hh^{n-k_0+1}}(s-\tfrac{k_0}{2})$:
\[ R_{X_c}(s)\sum_{I\in\mc{I}_0}\varphi_I\phi_I= 
x^{-\frac{n-k_0}{2}}R_{\hh^{n-k_0+1}}(s-\tfrac{k_0}{2})(x^{\frac{n-k_0}{2}}\sum_{I\in\mc{I}_0}\varphi_I\phi_I).
\]
We can then apply the expansion in Lemma \ref{expleft} for $R_{\hh^{n-k_0+1}}(s-\frac{k_0}{2})$.  This shows
that the operator $\Gamma(s-\ndemi+\ell+1)M_{\ell}(s)$ restricted to functions of the form
$\varphi=\sum_{I\in\mc{I}_0}\varphi_I\phi_I$ have poles at 
$k_0/2-k\in k_0/2-\nn_0$ with residues also given using Lemma \ref{expleft}, and rank $\mc{O}(k^{n-k_0})$.

The fact that $M_\ell(s)\varphi\in C^\infty(F)$ if $\varphi\in C_0^\infty(X_c)$ is a consequence of elliptic regularity and the fact that for all $L\in\nn_0$, 
$\Delta_F^LK_{s-\ndemi}(x' \sqrt{\Delta_F})\varphi(x',\cdot)$ is bounded in $L^2(F)$ uniformly for $x'$ in any compact set 
of $(0,\infty)$. Finally, the properties of $R_{X_c,N}(s)$ follows directly from those of $R_{X_c}(s)$  and 
of $M_{\ell}(s)$.
\end{proof}
\subsection{Cusp parametrix construction}\label{paramcusp}

For the parametrix in a model cusp $X_c$, we will consider the model neighbourhood
$\Gamma^c\backslash\{(x,y,z)\in \hh^{n+1} ; \>x^2+\abs{y}^2 \ge R\}$.
We need to introduce a rather complicated series of cutoff functions:
\begin{enumerate}
\item  $\chi \in \cinf(X_c)$ is independent of $\delta$, has support in the cusp region 
$\{x^2+\abs{y}^2  \ge R+2\}$, with $\chi=1$ in 
$\{x^2+\abs{y}^2  \geq R+3\}$.  This corresponds to the cusp neighbourhood 
component of a partition of unity for the set of charts described in \S\ref{setup.sec}.
\item $\psi \in \cinf_0(\bbar{X}_c)$ is supported away from the cusp.  (In the global parametrix construction this will correspond to a cutoff function $\psi\in C_0^\infty(\bbar{X})$.)  This $\psi$ is also independent of $\delta$, and we
require that $\psi=1$ on some neighbourhood of the support of $1-\chi$.
\item   The horizontal and vertical cutoffs $\phhor \in\cinf_0(\bbR^{n-k_0})$ and $\phvert_\delta \in \cinf(X_c)$. 
Here $\phhor$ depends only on $\abs{y}$ and is equal to $1$ for $\abs{y} \le R+1$ and $0$ for $\abs{y} \ge R+2$.
The function $\phvert_\delta$ depends on $(x,y)$, and its support is a bit more complicated, to accommodate the geometry of the cusp neighbourhood:  we assume that $\phvert_\delta = 1$ on $\{x\le \delta\} \cup \{x^2+\abs{y}^2 \le R\}$, and $\phvert_\delta = 0$ on $\{x\ge 2\delta\} \cap \{x^2+\abs{y}^2 \ge R+1\}$.  Furthermore, we assume that $\phvert_\delta$ depends only on $x$ for $R+1 \le \abs{y} \le R+2$.
(That is, the cutoff $\phvert_\delta$ is actually vertical where the supports of $\nabla \phhor$ and $\phvert_\delta$ intersect.)
\end{enumerate}

The structure of the supports of these cutoffs is illustrated in Figure~\ref{cutoffs.fig}.
As in the regular case, we assume that $\phhor$ and $\phvert_\delta$ satisfy quasi-analytic estimates of the form \eqref{QAE}.
The assumptions above guarantee in particular that $1-\phhor\phvert_\delta$ is supported in $\{x^2+\abs{y}^2 \ge R\}$ and equal to 1 on the support of $\chi$.  Furthermore, we have the horizontal/vertical decomposition of the commutator,
\[
[\Delta_{X_c}, \phvert_\delta\phhor] = x^2\phvert_\delta [\Delta_F, \phhor] + \phhor [\Delta_{X_c}, \phvert_\delta].
\]
The first term on the right is supported in $x\le 2\delta$, and the second is compactly supported in the interior, so we preserve the essential properties from the regular case. 

\begin{figure}
\begin{center}
\begin{overpic}{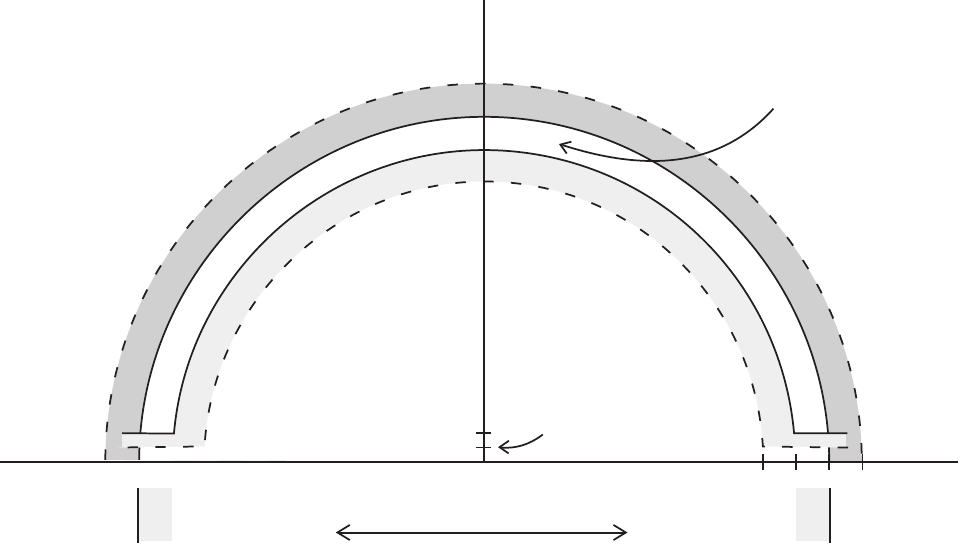}
\put(78,5){$\scriptstyle R$}
\put(57,11){$\scriptstyle\delta$}
\put(46,3){$\phhor = 1$}
\put(1,3){$\phhor = 0$}
\put(88,3){$\phhor = 0$}
\put(8,49){$\psi=1$}
\put(55,49){$\chi=1$}
\put(55,30){$\phvert_\delta=1$}
\put(81,45){$\begin{cases} \chi=0\\ \phvert_\delta=1\end{cases}$}
\put(99,10){$\scriptstyle y$}
\put(51,56){$\scriptstyle x$}
\end{overpic}
\end{center}
\caption{Structure of cutoffs in the model cusp neighbourhood (with the $z$ coordinate suppressed).  The supports of $\nabla\phhor$ and $\nabla\phvert_\delta$ are light gray and the support of $\nabla\chi$ is dark gray.  The entire picture is contained within the support of $\psi$.}\label{cutoffs.fig}
\end{figure}
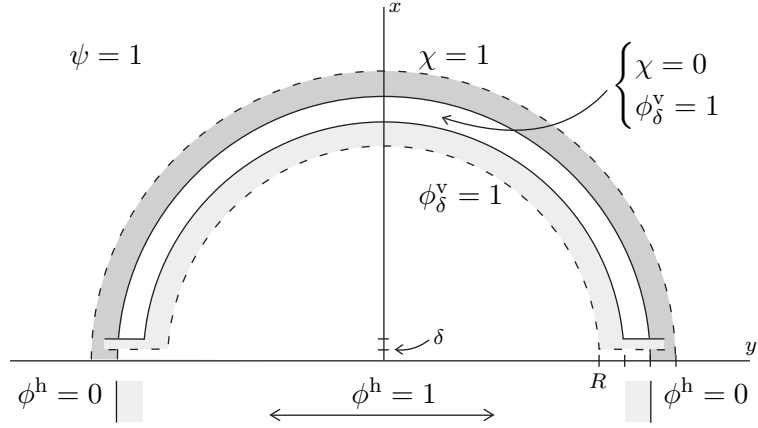

\bigbreak
Our initial parametrix is $(1-\phvert_\delta\phhor) R_{X_c}(s)\chi$, which satisfies
\[
\begin{gathered}
(\Delta_{X_c}-s(n-s))(1-\phvert_\delta\phhor) R_{X_c}(s)\chi  = 
\chi + K_{X_c,0}(s)+L_{X_c,0}(s), \\
 L_{X_c,0}(s):=-\phhor [\Delta_{X_c},\phvert_\delta]R_{X_c}(s)\chi, \\ 
 K_{X_c,0}(s):=-x^2 \phvert_\delta [\Delta_F,\phhor]R_{X_c}(s)\chi.
\end{gathered}
\]
From the fact that $\chi\nabla (\phvert_\delta\phhor)=0$ and \cite[Prop 5.3]{GM:2012},  
the Schwartz kernel of $K_{X_c,0}(s)\psi$, and $L_{X_c,0}(s)\psi$ satisfy
\[K_{X_c,0}(s)\psi \in x^{s+2}{x'}^sC_0^\infty(\bbar{X}_c\x\bbar{X}_c),\quad 
L_{X_c,0}(s)\psi \in x^{\infty}{x'}^sC_0^\infty(\bbar{X}_c\x\bbar{X}_c).\] 
The $L_{X_c,0}(s)\psi$ term is already in any Schatten class
on $\rho^NL^2(B_0,dg_0)$ for $\re s>n/2-N$, but the $K_{X_c,0}(s)\psi$ term is not, thus we need to 
improve the parametrix construction. 

Our basis for the estimates of boundary terms will be Lemma~\ref{QAEforRXc}.  To apply that result, set $\psi_2 = \chi\psi$
and let $\psi_1$ be some cutoff such that $\psi_1 = 1$ on the support of $[\Delta_{X_c},\phvert_\delta \phhor]$. 
Then we have, as Schwartz kernels,
\[
\psi_1 R_{X_c}(s;\cdot, \cdot) \chi\psi = F(s;\cdot, \cdot),
\]
satisfying the estimates given in Lemma~\ref{QAEforRXc}.  
Notice that $x \le 2\delta$ in the support of $\phvert_\delta [\Delta_F,\phhor]$.
We can then apply the boundary expansion of Lemma \ref{Mellcusp} to $x^{-s}F(s)$:
\[
-x^2 \phvert_\delta [\Delta_F,\phhor] x^{-s}F(s) = 
-  \phvert_\delta [\Delta_F,\phhor] \sum_{\ell=0}^{N-1}x^{s+2\ell+2}M_\ell(s)\chi\psi -
\phvert_\delta [\Delta_F,\phhor] F_{N}(s),
\]
where $F_{N}(s)$ comes from the remainder term from the Taylor expansion of $x^{-s}F(s)$ at $x=0$ and 
the operators $M_\ell(s)$ are considered as Schwartz kernels.

The next step is to apply Lemma~\ref{f.abc} to the boundary terms of the form
\begin{equation}\label{Deffjcusp}
f_j(y,z;w'):= [\Delta_F,\phhor(y)]M_{j-1}(s;y,z,w')\chi(w'),
\end{equation} 
with $w'\in X_c$ viewed as a parameter, there exist for some differential operators $\mc{A}_{j,N}(s), \mc{B}_{j,N}(s)$
with smooth coefficients on $\bbR^+\x F$, such that for $j=1,\dots,N$,
\[
(\Delta_{X_c}-s(n-s)) x^{s+2j}\mc{A}_{j,N}(s)f_j = 
x^{s+2j}f_j + x^{s+ 2N+2}\mc{B}_{j,N}(s)f_j,
\] 
 where the term on the right-hand side has Schwartz kernel in 
$x^{s+2N+2}{x'}^sC_0^\infty(\bbar{X}_c\x \bbar{X}_c)$.
Furthermore,
\[\frac{\mc{A}_{j,N}(s)}{\Gamma(s-n/2+j)} \textrm{ and }\frac{\mc{B}_{j,N}(s)}{\Gamma(s-n/2+j)}\textrm{ are holomorphic in }s.\]

To conclude, we set 
\[
Q_{X_c,N}(s) := (1-\phvert_\delta\phhor) R_{X_c}(s)\chi - \phvert_\delta\sum_{j=1}^{N}x^{s+2j} \mc{A}_{j,N}(s) 
[\Delta_F,\phhor]M_{j-1}(s)\chi,
\]
which leads to an error term,
\begin{equation}\label{ENcusp}
\begin{split}
E_{X_c,N}(s) &:=  -\phvert_\delta [\Delta_F,\phhor] F_{N}(s) - \phvert_\delta\sum_{j=1}^{N}x^{s+2N+2}\mc{B}_{j,N}(s) [\Delta_F,\phhor]M_{j-1}(s)\chi \\
&\qquad - [\Delta_{X_c},\phvert_\delta] \left(\phhor F(s)
+ \psi_1\sum_{j=1}^{N}x^{s+2j}\mc{A}_{j,N}(s) [\Delta_F,\phhor]M_{j-1}(s)\chi \right). \\
\end{split}
\end{equation}

\begin{prop}\label{paracusp}
Let $\chi,\psi$ be cutoff functions as explained just above. Let $N\in\nn$ be large, then the operators $Q_{X_c,N}(s),E_{X_c,N}(s)$, defined above, satisfy
\[ (\Delta_{X_c}-s(n-s))Q_{X_c,N}(s)=\chi+E_{X_c,N}(s),\] 
with $Q_{X_c,N}(s),E_{X_c,N}(s)$ meromorphic in $\{\re s>n/2-N\}$ with simple poles at $k_0/2-k$ with $k\in \nn_0$
and with residue an operator of rank $\mc{O}(k^{n-k_0})$. 
The operator $\psi Q_{X_c,N}(s)\psi :\rho^NL^2(X_c)\to \rho^{-N}L^2(X_c)$ is bounded
 for $\re s>n/2-N$ and $s\notin (n/2-\nn/2)$, and the Schwartz kernel 
of the error term $E_{X_c,N}(s)\psi$ can be written as
\[
E_{X_c ,N}(s;\omega,\omega')\psi(\omega') =  [\Delta_{X_c},\phvert_\delta] h_{\rm cpt}(s;\omega,\omega') + \phvert_\delta x^{s+2N+2} h_N(s;\omega,\omega'),
\]
where $h_{\rm cpt}, h_N$ are  smooth functions defined in $\supp( \phhor \nabla\phvert_\delta)\x \supp(\chi\psi)$ and 
$\{x<3\delta\}\x\supp(\chi\psi)$, respectively.   On these domains they satisfy the derivative bounds 
\begin{equation}\label{EN.bdcusp}
\begin{gathered}
\norm{\pl_{w}^\alpha h_N(s;\cdot,\cdot) {\rho'}^N}_{L^\infty} \le 
C^{\abs{\alpha}+N} N^{\abs{\alpha}} e^{C\brak{s}} e^{\Lambda_{\Gamma^c}([\re s]_-)}\delta^{-[\re s]_-}, \\
\norm{\pl_{w}^\alpha h_{\rm cpt}(s;\cdot,\cdot) {\rho'}^N}_{L^\infty} \le 
C_\delta^{\abs{\alpha}} N^{\abs{\alpha}}  e^{\Lambda_{\Gamma^c}([\re s]_-)} \delta^{-2[\re s]_-}e^{C(|\im s|+N)-c\re s}
\end{gathered}
\end{equation}
for $s\in U_N\cap\{{\rm dist}(s,n/2-\nn/2)>\eps\}$, $\abs{\alpha}\leq 6N$, and where $C$ is independent of $s,N,\alpha,\delta$, $C_\delta$ independent of $s,N,\alpha$.
\end{prop}  
\begin{proof}
Since $f_j$ and $F_{N}$ are the Taylor coefficients and remainder, respectively, for the expansion of $x^{-s}F$ at $x=0$, the quasi-analytic estimates of $x^{-s}F(s)$ from Lemma~\ref{QAEforRXc} carry over to these terms.   We can then use the combination of Lemma~\ref{QAEforRXc} and Lemma \ref{f.abc}, together with the 
derivative bounds \eqref{derivex^s} applied with $s+2j$ instead of $s$, we deduce the claimed estimates on
\[
h_{\rm cpt}(s) := - \phhor F(s)\psi - \psi_1\sum_{j=1}^{N}x^{s+2j}\mc{A}_{j,N}(s) f_j(s) \psi
\]
which is well defined on $\supp( \phhor \nabla\phvert_\delta)\x \supp(\chi\psi)$ and 
\[
h_N(s) :=  -[\Delta_F,\phhor] x^{-s-2N-2}F_{N}(s)\psi 
-\sum_{j=1}^{N} \bigl(\mc{B}_{j,N} f_j \bigr)(s) \psi.
\]
Note that we can assume that $\psi_1$ is chosen so that $x \le 2\delta$ in the support of $\psi_1 [\Delta_{F}, \phhor]$,
so for $\re s \ge n/2$ we pick up an extra factor $\delta^{\re s}$ from the $f_j$ term in the $h_{\rm cpt}(s)$ estimate, which accounts for the $e^{-c\re s}$ term from the $j$ sum in $h_{\rm cpt}$ 
(assuming that $\delta$ is sufficiently small, we can assume that the decay of this term is dominated by the $e^{-c\re s}$ when $\re s>0$).
\end{proof}

\bigbreak
We also need a cusp neighbourhood version of Proposition~\ref{GZ.prop}:
\begin{prop}\label{GZ.propcusp}
For $\delta, \rho, \chi$ as above and for any $N\in \nn$
there exist meromorphic families $S_{X_c,N}(s)$, $K_{X_c,N}(s)$ and $L_{X_c,N}(s)$ of 
operators with poles at $k_0/2-\nn_0$ of finite rank, such that
\[
\begin{split}
\psi S_{X_c,N}(s)\psi &:\rho^NL^2(X_c)\to \rho^{-N}L^2(X_c),\\
K_{X_c,N}(s)\psi &:\rho^NL^2(X_c)\to \rho^NL^2(X_c), \\
L_{X_c,N}(s)\psi  &:\rho^NL^2(X_c)\to \rho^NL^2(X_c), \\
\end{split}
\] 
for $\re s > n/2-N$, where $\psi\in C_0^\infty(\bbar{X}_c)$, such that
\[(\Delta_{g_0}-s(n-s))S_{X_c,N}(s)=\chi+K_{X_c,N}(s)+L_{X_c,N}(s), 
\]
and, the following boundedness hold for $\re s>n/2-N$ and $s\notin k_0/2-\nn_0$

In addition, the operators $K_{X_c,N}(s)\psi, L_{X_c,N}(s)\psi$ 
are trace class in $\rho^NL^2(X_c)$ and for $\delta$ sufficiently small have singular values satisfying for some constants $c_\delta, C, B_\delta$ independent of $s,N$, 
\begin{equation}\label{cusp.muK}
\mu_j(K_{X_c,N}(s)\psi)\leq C e^{\Lambda_{\Gamma^c}([\re s]_-)} e^{-c_\delta N} j^{-2},  \quad j \ge1,
\end{equation}
for $s \in U_N\cap \{{\rm dist}(s, k_0/2-\nn_0)>\eps\}$, and
\begin{equation}\label{cusp.muL}
\mu_j(L_{X_c,N}(s)\psi)\leq  e^{\Lambda_{\Gamma^c}([\re s]_-)} \begin{cases} C_\delta^N & \text{for }j\ge 1,\\
Ce^{-cN} j^{-2}& \text{for }j \ge B_\delta N^n \end{cases}
\end{equation}
for $s \in U_N\cap \{{\rm dist}(s, k_0/2-\nn_0)>\eps\}$.  Moreover, assuming $\delta$ sufficiently small and $\gamma$ sufficiently large, for $s_N \ge \gamma N$ we have the estimate
\begin{equation}\label{normLcusp}
\norm{L_N(s_N)}_{\rho^NL^2}\leq e^{-c_{\delta} N}.
\end{equation}

The operators $K_{X_c,N}(s)$ and $L_{X_c, N}(s)$ have possibly poles at 
$s = k_0/2-k$ for each $k \in \bbN_0$, and the polar part in the Laurent expansion are
some operators of rank bounded by $\mc{O}(k^{n-k_0})$.
\end{prop}
\begin{proof} 
The proof is similar to that of Proposition~\ref{GZ.prop}. 
Using the cutoff $\chi$ we first construct $Q_{X_c,N}(s),E_{X_c,N}(s)$ as in Proposition~\ref{paracusp}, satisfying 
\[
(\Delta_{X_c}-s(n-s))Q_{X_c,N}(s)=\chi+E_{X_c,N}(s).
\]
Then we introduce a second cutoff $\tilde\chi \in C^\infty(X_c)$ such that $\tilde\chi = 1$ on 
$\{x^2+\abs{y}^2 \ge R/2\}$, and perform the construction again to produce $\tilde{Q}_{X_c,N}(s),
\tilde{E}_{X_c,N}(s)$.  The original $Q_{X_c,N}(s)$ is then replaced by
\[
S_{X_c,N}(s;\cdot,\cdot) = Q_{X_c,N}(s;\cdot,\cdot) -  \tilde{Q}_{X_c,N}(s) [\Delta_{X_c},\phvert_\delta] x^{s} h_{\rm cpt}(s;\cdot,\cdot).
\]

The error terms become
$$
K_{X_c,N}(s;\cdot,\cdot) = \phvert_\delta x^{s+2N+2} h_N(s;\cdot,\cdot) 
$$
and 
$$
L_{X_c,N}(s; \cdot,\cdot) = - \tilde{E}_{X_c,N}(s) [\Delta_{X_c},\phvert_\delta] x^{s} h_{\rm cpt}(s;\cdot,\cdot),
$$
where $h_N$, $h_{\rm cpt}$ are defined in Proposition~\ref{paracusp}.  Now, we can replace 
$\phvert_\delta$ with $H_\delta$, the characteristic function of a region whose boundary $\Sigma$ 
is a hypersurface interpolating between the sets 
$\{x= \delta,\> R+1 \le x^2+\abs{y}^2 \le R+2 \} \cup \{x\ge 2\delta, \> x^2+\abs{y}^2 = R\}$ as illustrated in Figure~\ref{Sigma.fig}.  In order to
preserve the derivative estimates that follow from \eqref{QAE} in the case of smooth $\phvert_\delta$, we assume that 
$\Sigma$ is the graph over $\{\abs{y}^2\le R+2\}$ of a function satisfying quasi-analytic derivative bounds analogous to \eqref{QAE}, for derivatives of up to order $10N$.  
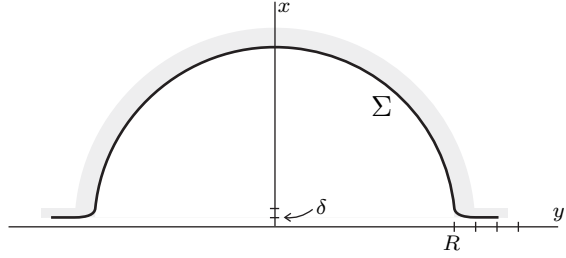
\begin{figure}
\begin{center}
\begin{overpic}{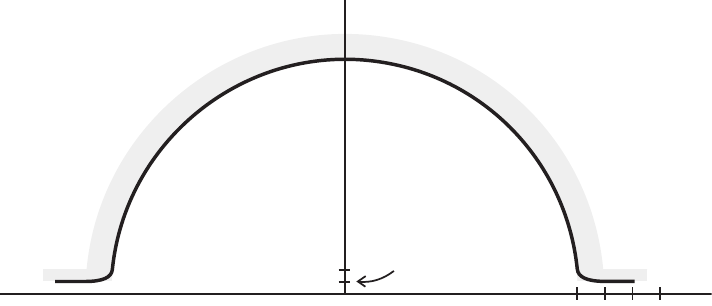}
\put(79,-2.5){$\scriptstyle R$}
\put(56,4){$\scriptstyle \delta$}
\put(99,3.5){$\scriptstyle y$}
\put(49,41){$\scriptstyle x$}
\put(66,22){$\Sigma$}
\end{overpic}
\end{center}
\caption{The hypersurface used to define $H_\delta$ (with the original support of $\nabla \phvert_\delta$ in gray).}\label{Sigma.fig}
\end{figure}

Analysis of these terms now works essentially as in the proof of Proposition~\ref{GZ.prop}.  
For the $K_{X_c,N}(s)$ we insert $\Delta_w^{n+1}$ and derive the decay of the singular values from \eqref{mujK}.  
From the prefactor $x^{s+2N+2}\rho^{-N}$ we gain an extra factor of $\delta^{\re s+N}$, for the support of $K_{X_c,N}(s)$ lies in $x \le 3\delta$.
In this way we obtain the estimate
\[
\mu_j(K_{X_c,N}(s)) \le C^N e^{C\brak{s}} \delta^{\re s + N-[\re s]_-} e^{\Lambda_{\Gamma^c}([\re s]_-)}.
\]
Taking $s \in U_N$ and choosing $\delta$ sufficiently small yields the factor $e^{-c_\delta N}$.
  
For the $L_{X_c,N}(s; \cdot,\cdot)$ term we use the fact that $\phhor [\Delta_{X_c}, H_\delta]$ is a distribution supported on a compact hypersurface $\Sigma$ to introduce a comparison to Dirichlet eigenvalues on $\Sigma$.  This hypersurface now has a more complicated geometry, 
but the assumption that $\Sigma$ is the graph of a quasi-analytic function allows us to estimate after inserting high powers 
of $\Delta_\Sigma$ just as before.  Beyond this, all that matters for the singular value estimate is that $\dim \Sigma= n$.

Finally, consider the norm estimate \eqref{normLsN}.  For $\re s$ large we have a factor of $e^{-c\re s}$ from the $h_{\rm cpt}(s)$ estimate in \eqref{EN.bdcusp}.  And $\tilde{E}_{X_c,N}(s)$ is built from components that also satisfy \eqref{EN.bdcusp}, i.e.~a component supported near the boundary that gains a factor $\delta^{\re s}$ from the support restriction and an interior component with decay like that of $h_{\rm cpt}(s)$.  Thus, assuming $\delta$ sufficiently small, we can estimate
\[
\norm{L_{X_c,N}(s_N)} \le e^{CN - c \re s_N},
\]
for real $s_N > n/2$. All the estimates have been done outside an $\eps$ neighbourhood of $n/2-\nn/2$, but by using the maximum principle and the holomorphy of the operators outside $k_0/2-\nn_0$, we deduce directly the bounds in 
$U_N\cap \{{\rm dist}(s, k_0/2-\nn_0)>\eps\}$. 
\end{proof}

\bigbreak
\section{Global parametrix construction}\label{gl.par.sec}

We return now to the global case of a geometrically finite quotient, $X=\Gamma\backslash\hh^{n+1}$.  For $j=1,\dots n_c$, we denote by  $\Gamma_j^c\subset \Gamma$  a finite set of representatives of conjugacy classes of  
parabolic subgroup, each one corresponding to a cusp of rank $k_j$. 
As we mentioned before, they can be assumed to be abelian after passing to a finite cover. 

For the covering outlined in \S\ref{setup.sec}, consisting 
of the relatively compact $\mc{U}_0$, regular boundary and cusps neighbourhoods $\{\mc{U}_j\}_{j\in J^r\cup J^c}$, we introduce a corresponding partition of unity, 
\[
1 = \chi_0 + \sum_{j\in J^r\cup J^c} \chi_j. 
\]
For a small parameter $\delta>0$ each $\chi_j$ with $j\in J^r$ is assumed to satisfy the local coordinate assumptions placed on $\chi_\delta$ in \S\ref{paramreg}.  Likewise, $\chi_j$ for  $j\in J^c$ is assumed to satisfy the assumptions placed on $\chi$ in \S\ref{paramcusp}.
We have a global cutoff $\psi \in C_0^\infty(\bbar{X})$, which we can assume to equal 1 within $\mc{U}_0$ and each
$\mc{U}_j$, $j\in J^r$.  In $\mc{U}_j$ for $j\in J^c$ we assume that $\psi = 1$ on the support of $1 - \chi_j$.

The boundary defining function $\rho$ is also assumed to satisfy the local ($\delta$-dependent) assumptions of Lemma \ref{constrho}.  Those assumptions were imposed in particular to insure that
\[
\rho \chi_0 =1, 
\]
so that the factors $\rho^{\pm N}$ have no impact on the interior parametrix term when considering operators in $L^2$ weighted space.

For the `near boundary' parametrix terms, we apply the constructions from Propositions~\ref{GZ.prop} and \ref{GZ.propcusp} to obtain
local parametrices such that for each $j \in J^c \cup J^r$,
\[
(\Delta_{X} - s(n-s)) S_{\mc{U}_j,N}(s) = \chi_j + K_{\mc{U}_j,N}(s) + L_{\mc{U}_j,N}(s),
\]
with operators having the properties given in those lemmas.

For the interior parametrix, we take  some $\hat\chi_0 \in \cinf_0(X)$ such that $\hat\chi_0 = 1$ on the support of $\chi_0$ while still $\rho \hat\chi_0 = 1$.  For $s_N\gg n$, which will be specified later, we define 
\begin{equation}
M_0(s_N):=\hat{\chi}_0R_X(s_N)\chi_0,
\end{equation}
where $R_X(s)=(\Delta_X-s(n-s))^{-1}$ is the resolvent in the physical half-space $\re s>\ndemi$. This gives 
\begin{equation}\label{errorcpt}
\begin{gathered}
(\Delta_X-s(n-s))M_0(s_N)=\chi_0+K_0(s,s_N) \,\,, \textrm{ with  }\\
K_0(s,s_N):=[\Delta,\hat{\chi}_0]R_X(s_N)\chi_0+ (s_N(n-s_N)-s(n-s))M_0(s_N).
\end{gathered}
\end{equation}

The full parametrix is
\[
M_N(s) := M_0(s_N) + \sum_{j\in J^r\cup J^c} S_{\mc{U}_j,N}(s), 
\]
which for $\re s \ge n/2-N$ satisfies,
\[
(\Delta_X-s(n-s))M_N(s) = I + E_N(s),
\]
where
\[
E_N(s)  := K_0(s,s_N) + K_N(s), 
\]
with
\[
K_N(s) := \sum_{j\in J^r\cup J^c} \Bigl[ K_{\mc{U}_j,N}(s) + L_{\mc{U}_j,N}(s) \Bigr].
\]
Under the assumption on $\psi$, then $\psi E_N(s) = E_N(s)$ and we can insert the cutoff to obtain
\[
(\Delta_X-s(n-s))M_N(s)\psi = \psi(I + E_N(s)\psi).
\]

We can estimate the singular values of $E_N(s)\psi$ using Propositions~\ref{GZ.prop} and \ref{GZ.propcusp}.  The cusp estimates contain extra factors related to Diophantine approximation, as introduced in \S\ref{Dioph.sec}.  To cover estimates for the full space we introduce
\[
\Lambda_X(u) := \sup_{j\in J^c} \Lambda_{\Gamma^c_j\backslash \hh^{n+1}}(u),
\]
where $\Lambda_{\Gamma^c}(s)$ was the growth function defined in \eqref{Lambda.def}.

\begin{lemma}\label{ENsN.norm}
For $s_N := aN$ with $a>0$ sufficiently large and independent of $N$,
\[
\norm{E_N(s_N)\psi}_{\rho^NL^2} \le \frac12.
\]
\end{lemma}

\begin{proof}
For the $K_N(s)$ term, we see from \eqref{reg.muK}, \eqref{normLsN}, \eqref{cusp.muK}, and \eqref{normLcusp} give the estimate
\[
\norm{K_N(s_N)\psi}_{\rho^NL^2} \le Ce^{-cN},
\]
for $s_N$ as above.  And since $K_0(s_N,s_N) = [\Delta,\hat{\chi}_0]R_X(s_N)\chi_0$ we have the standard spectral estimate (recall that $\rho=1$ on the supports of $\hat{\chi}_0,\chi_0$),
\[
\norm{[\Delta,\hat{\chi}_0]R_X(s_N)\chi_0}_{\rho^NL^2} = \mc{O}(1/s_N).
\]
\end{proof}

In particular, with such a choice of $s_N$, $I - E_N(s_N)\psi$ is invertible, so $I - E_N(s)\psi$ is meromorphically invertible by the analytic Fredholm theorem, with poles of finite rank.  
This gives the (cutoff) resolvent as a meromorphic family in $\re s>n/2-N$   
\begin{equation}\label{Mpsi.Epsi}
R_{X_c}(s)\psi = M_N(s)\psi (I + E_N(s)\psi)^{-1}.
\end{equation}

\subsection{Determinant estimates}

We can now proceed to estimate the resonances by applying the Fredholm determinant method just as adapted by 
Guillop\'e-Zworski \cite{GZ:1995b}.  Note that $E_N(s)\psi$ is the sum of a pseudodifferential operator of order $-1$, with compactly supported coefficients, plus a smoothing operator which is trace class on $\rho^N L^2$ for $\re s> n/2-N$.  Hence $E_N(s)^{n+2}$ is trace class on $\rho^N L^2$ and we can form the determinant in this space
\[
D_N(s) := \det \left[I - \bigl(-E_N(s)\psi)^{n+2} \right].
\]
By Propositions~\ref{GZ.prop} and \ref{GZ.propcusp}, $E_N(s)\psi$ has possible finite order poles at 
$s = n/2-k/2$ for $k \in \bbN$, with the polar part in the Laurent expansion an operator of rank bounded by $\mc{O}(k^{n})$.
Thus, by \cite[Lemma~A.1]{GZ:1995b}, $D_N(s)$ has poles at $n/2-k/2$ for $k \in \bbN$ with orders bounded by 
$Lk^{n}$ for some  $L\in\nn$ independent of $k,N$.

To cancel these poles we introduce the canonical product
\[
g_L(s) := s^{L} \prod_{k\in\bbN} \prod_{\omega\in U_{2(n+1)}} E\left(-\omega \frac{2s}{k}, n+1\right)^{2L k^n},
\]
where $E(z,p) := (1-z) e^{z+\dots+z^p/p}$ and $U_m$ denotes the set of $m$-th roots of unity. The inclusion of rotations 
by roots of unity guarantees, 
by Lindel\"of's theorem \cite[Thm.~2.10.1]{Boas}, that $g_L$ is of finite type.  Thus, we have the order estimate:
\begin{equation}\label{g.order}
\log \abs{g_L(s)} = \mc{O}(\brak{s}^{n+1}).  
\end{equation}
By the choice of $L$ as indicated above, the function,
\[
h_N(s) := g_L(s) D_N(s),
\]
will be holomorphic for any $N$.

By \eqref{Mpsi.Epsi}, and the arguments used for \cite[Lemma~3.2]{GZ:1995b}, the resonances of $X$ are contained among the zeros $\zeta$ of $D_N(s)$, with multiplicities $m_\zeta(D_N)$, and the set $n/2-\bbN/2$ with multiplicity of 
$(n-k)/2$ for $k \in \bbN$ bounded by 
$\mc{O}(k^{n})$.  Hence to prove Theorem~\ref{maintheorem} it suffices to prove the corresponding estimate for the zeros of $h_N(s)$.

To estimate the growth of $h_N(s)$ we introduce a combined Diophantine growth function,
\[
\Lambda_X(u) := \max_{j\in J^c} \Lambda_{\Gamma^c_j\backslash \hh^{n+1}}(u),
\]
where $\Lambda_{\Gamma^c}(s)$ was the growth function defined in \eqref{Lambda.def}.
\begin{prop}\label{hN.growth}
For $s \in U_N$ with $\abs{s} \le CN$ and such that $\Lambda_X([\re s]_-) \le N$, we have
\[
|h_N(s)| \le e^{CN^{n+1}}.
\]
\end{prop}
\begin{proof}
For the error terms $K_{\mc{U}_j,N}(s)$ and $L_{\mc{U}_j,N}(s)$ we have singular value estimates from Propositions~\ref{GZ.prop} and \ref{GZ.propcusp} for $s \in U_N$.  Under the extra assumption that $\Lambda_X([\re s]_-) \le N$, these can be combined via the Fan Inequalities for singular values, to give
\[
\mu_l(K_N(s)\psi) \le
\begin{cases}
e^{CN}, & \\
Ce^{-cN} l^{-2}, &\text{for } l \ge BN^n.
\end{cases}
\]
A simple estimate based on Weyl's determinant inequality then shows that
\[
\det \left( 1 + \abs{K_N(s)\psi} \right) \le e^{CN^{n+1}}.
\]

The interior error term $K_0(s,s_N)$ is the sum of two compactly supported components.  
Since $[\Delta,\hat{\chi}_0]R_X(s_N)\chi_0$ is order $-1$ and $M_0(s_N)$ has order $-2$, comparison to the Laplacian on a a compact domain gives the basic estimate
\[
\mu_k(K_0(s,s_N)) = \mc{O}(k^{-1/(n+1)})+ \mc{O}(\brak{s}^2 k^{-2/(n+1)}).
\]
In this case Weyl's estimate gives
\[
\det\left( 1+ \abs{K_0(s,s_N)}^{n+2} \right) \le e^{C\brak{s}^{n+1}}.
\]

The two determinant estimates are combined using \cite[Lemma~6.1]{GZ:1995a} to give the growth estimate
\[
\abs{D_N(s)} \le e^{CN^{n+1}},
\]
for $s \in U_N$ with $\abs{s} \le CN$ and $\Lambda_X([\re s]_-) \le N$.  In conjunction with \eqref{g.order}, this proves the growth estimate for $h_N(s)$. 
\end{proof}

\subsection{Proof of Theorem \ref{maintheorem}}
To apply Jensen's formula we also need a lower bound on $h_N(s_N)$ for $s_N := a N$.  
From \cite[Lemma~5.1]{GZ:1995b} we already know that
\[
\log \abs{g_L(s)} \ge - C_\vep \brak{s}^{n+1},
\]
for $\dist(s, U_{2(n+1)}\cdot\bbN) > \vep$.  The norm bound from Lemma \ref{ENsN.norm} implies that $D_N(s_N)$ is bounded below by
a constant independent of $N$.  Thus we have the lower bound
\begin{equation}\label{hNsN.bound}
\log\abs{h_N(s_N)} \ge - CN^{n+1}.
\end{equation}

\begin{figure}
\begin{center}
\begin{overpic}{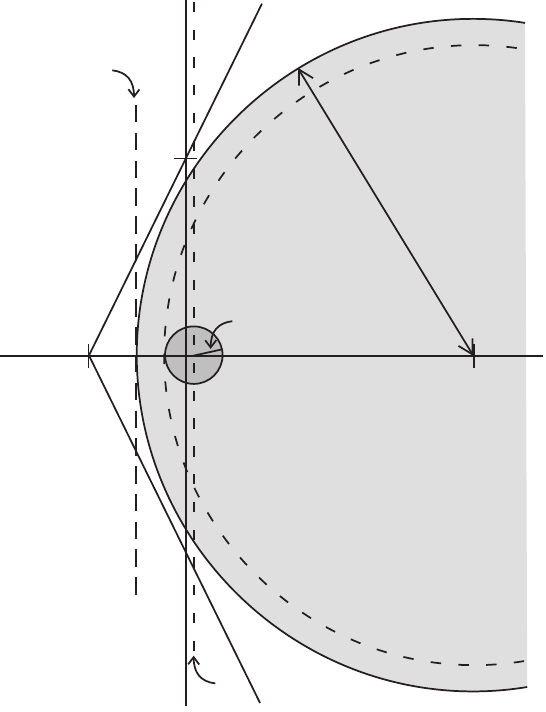}
\put(2,45){$\scriptstyle -N/4$}
\put(65,46){$\scriptstyle s_N$}
\put(52,76){$\scriptstyle s_N+2T$}
\put(33.5,54){$\scriptstyle T$}
\put(30.5,3){$\tfrac{n}2$}
\put(-8,89){$\scriptstyle \Lambda_X(\re s) = N$}
\end{overpic}
\end{center}
\caption{The big disk in which Jensen's inequality is applied, centered at $s_N$, containing the disk of radius $T$ centered at $n/2$ in which the resonances are counted.}\label{Jensen.fig}
\end{figure}
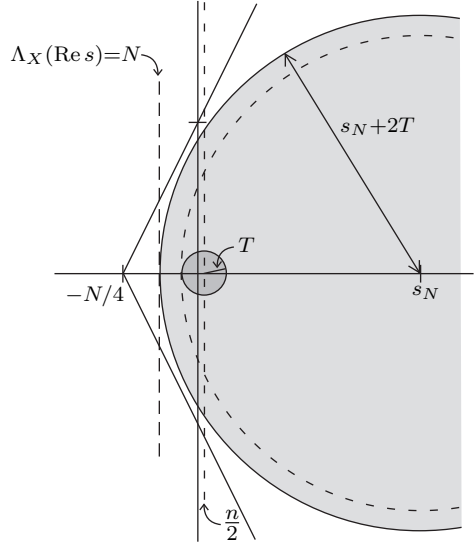
For $T>1$ large, we take $N$ to be  $N:=[\Lambda_X(2T)]$, so that the region covered by Proposition~\ref{hN.growth} includes a disk of radius $s_N+2T$ centered on $s_N$, as illustrated in Figure~\ref{Jensen.fig}.   
Note that $\Lambda_X(u) \ge c\brak{u} \log \brak{u}$, so that $T$ will be $o(N)$.
Let $\mc{N}(f; s_0,t)$ denote the number of zeros of $f(s)$ with $\abs{s - s_0} \le t$.  Using Proposition~\ref{hN.growth} and \eqref{hNsN.bound}, applying Jensen's formula to the big disk centered at $s_N$ gives
\[
\begin{split}
\int_0^{s_N+2T} \frac{\mc{N}(h_N; s_N,t)}{t}\>dt & \le \max_{\abs{s - s_N} = s_N+2T} \log \frac{\abs{h_N(s)}}{\abs{h_N(s_N)}} \\
&\le CN^{n+1}.
\end{split}
\]
We thus obtain the bound,
\[
\mc{N}(h_N; s_N,s_N+T) \le \frac{s_N+2T}{T}\int_{s_N+T}^{s_N+2T} \frac{\mc{N}(h_N; s_N,t)}{t}\>dt \le \frac{C N^{n+2}}{T}.
\]
Since a disk of radius $T$ centered at $n/2$ is contained in the region $\{\abs{s - s_N} \le s_N+T\}$, as shown in Figure~\ref{Jensen.fig},
this count gives the upper bound
\begin{equation}\label{NhN.bnd}
\mc{N}(h_N; n/2, T) = \mc{O}(N^{n+2}T^{-1}).
\end{equation}
As noted in the comments preceding Proposition~\ref{hN.growth},
\[
N_X(t) \le \mc{N}(h_N; n/2, t) + \mc{O}(t^{n+1}).
\]
Thus since $\Lambda_{X}(T)\gg T$ for large $T$, we get
\[
N_X(T) \le \frac{C(\Lambda_X(2T))^{n+2}}{T}.
\]

This completes the proof of Theorem~\ref{maintheorem}.  As a final remark, we note that that for a conformally compact quotient the same argument applies except that there is no need to restrict to $\Lambda_X([\re s]_-) \le N$ in Proposition~\ref{hN.growth}.  We could thus take $T \propto N$ in this case and then \eqref{NhN.bnd} would imply $N_X(T) = \mc{O}(T^{n+1})$.  As noted in the introduction, this is a new proof for the conformally compact case, in that it avoids the scattering determinant estimates used in \cite{Borthwick:2008}.

\subsection{Counting resonances in a strip}
Let us now consider the resonance count in a vertical strip:
\[
N_X(K,T) := \#\bigl\{s \in \calR_X \cap ([n/2-K,n/2]+i[0,T])\bigr\}.
\]
According to the fractal Weyl conjecture \cite{Sjostrand:1990, LSZ:2003}, we would expect $N_X(K,T)$ to satisfy a power law $\sim T^{1+\delta}$ as $T\to \infty$, with $\delta \in [0,n)$ the dimension of the limit set of $\Gamma$.  An upper bound with the expected exponent was proven for Schottky groups by Guillop\'e-Lin-Zworski \cite{GLZ:2004}, and recently extended to all convex cocompact $\Gamma$ (and even to non-constant curvature) by Datchev-Dyatlov \cite{DD:2012}.

Our estimate here is rather far from optimal.  The main point of interest is that it is independent of the Diophantine approximation problem.

\begin{prop}\label{Nstrip.prop}
For $K>0$ fixed, there exists $C_K$ depending on $K$ such that for all $T>1$
\[
N_X(K,T) \leq C_K T^{n+2}.
\]
\end{prop}
\begin{proof}
First we remark that the estimates we have proved imply that $ \re(D_N(s))>1/2$
for $s\in aN+i[0,\eps N]$, if $\eps>0$ is small enough and $a$ is large (both independent of $N$).   We can obtain a holomorphic function in $\re s>n/2-K$ whose zeros contain the resonances with multiplicities in $\re s>n/2-K$ by setting
\[
\til{D}_N(s) := D_N(s) \prod_{j=1}^{K+1}\frac{s-n/2+j}{s+K+1}. 
\]
Moreover, $\re (\til{D}_N(s))>1/2$
for $s\in aN+i[0,\eps N]$, if $N$ is large, for $K$ fixed. Let $\eps_0\in(0,1)$ so that $\til{D}_N(s)$ has no zeros 
on $\{\im s=-\eps_0, \re s\in[n/2-K-1,n/2+aN]\}$.

Let $\nu(\sigma,T)$ be the number of zeros of $\til{D}_N(s)$ in the rectangle $[n/2+\sigma,n/2+aN]+i[-\eps_0,T]$ where $T\in [-\eps_0,\eps N]$ and $\sigma\in [-K,aN]$.  
Then the Littlewood Lemma (see \cite[\S9.9]{Tit} or \cite[Prop 4.1]{Naud:2012}) gives  us the bound 
\[ \begin{split}
\int_{-K-1}^{aN} \nu(\sigma,T)\>d\sigma &\leq  C \int_{0}^T \left( \log \abs{\til{D}_N(n/2-K-1+it)} +\log \abs{\til{D}_N(n/2+aN+it)}\right) dt \\
&\qquad + C \int_{-K-1}^{aN} \abs{\arg(\til{D}_N(\sigma+iT))}+\abs{\arg(\til{D}_N(\sigma-i\eps_0))} d\sigma.
\end{split}\]
The function $\sigma\to \nu(\sigma,T)$ is decreasing as $\sigma$ increases and the left-hand side of the inequality 
is thus bounded below by  $\nu(-K,T)$. The bound  
\[ \log|\til{D}_N(s)| \leq CN^{n+1}\]
holds in the rectangle $[n/2+\sigma,n/2+aN]+i[-\eps_0,T]$.    What remains is to get an estimate for the argument of 
$\til{D}_N(\sigma+iT)$ and $\til{D}_N(\sigma-i\eps_0)$. Using the proof of Lemma 9.4 in \cite{Tit}, 
this follows from Jensen's formula and the fact 
that  $\abs{\arg(\til{D}_N(aN+iT))} <\pi/2$ and $\abs{\arg(\til{D}_N(aN-i\eps_0))}<\pi/2$.  We thus obtain, by setting $T=\eps N$,
\[ N_X(K,T) \leq \nu(-K,T) \leq C_K T^{n+2}.\]
where $C_K$ depends only on $K$.
\end{proof}
 
\bigbreak
\appendix
\section{Solving away boundary terms}\label{bnd.solv}

In this appendix, we give the details of the estimates on the terms that appear when we solve away the leading terms in Taylor expansions at the boundary.  We begin with a boundary solution lemma that applies either to $\Hyp$ or to a model cusp $X_c$.  In either case we study a model space $X = \bbR_+ \times F$, where $F$ is Euclidean $\bbR^{n}$ in the case of $\Hyp$, and a flat bundle of rank $n-k$ over a $k$-dimensional torus in the case of a rank $k$ cusp.  The Laplacian is given by
\[
\Delta_X = - (x\del_x)^2 + nx\del_x + x^2 \Delta_F.
\]

\begin{lemma}\label{f.abc}
Let $A>2$ and let $N\in\nn$ be large. 
For $j=1, \dots, N$ there exist differential operators $\cA_{j,N}(s)$, $\cB_{j,N}(s): \cinf_0(F) \to \cinf(X)$, such that for $f(s; \cdot, \cdot) \in \cinf_0(F)$,
$$
(\Delta_{X} - s(n-s)) x^{s+2j} \cA_{j,N}(s) f + x^{s+2j} f
= x^{s+2N+2} \cB_{j,N}(s) f.
$$
Let $\Omega\subset \{\re s>n/2-N\}$. If $f$ satisfies the quasi-analytic derivative estimates
\begin{equation}\label{df.quasi}
||\pl^{\alpha} f||_{L^\infty} \le C_0^{N+\abs{\alpha}} N^{\abs{\alpha}}, \quad s\in \Omega, \, \, 
\dist(s-n/2, -\bbN) > \vep,  \,\, 
\abs{\alpha}\leq AN
\end{equation}
for some constant $C_0$,  then $\cA_{j,N}(s) f$ and $\cB_{j,N}(s) f$ satisfy the same estimates 
as \eqref{df.quasi} with a new constant $C:=C_{\vep}C_0$ instead of $C_0$, 
where $C_\eps>0$  does not depend on $j$, $\Omega$ or $N$, and with $A$ replaced by $A-2$.
\end{lemma}

\begin{proof}
We start from the observation that for $a \in \cinf(F)$,
\[
(\Delta_{X} - s(n-s)) x^{s+2k} a + 4k(s-n/2+k) x^{s+2k} a   = x^{s+2k+2} \Delta_F a.
\]
We can thus set
$$
\cA_{j,N}(s) := \sum_{k=0}^{N-j} x^{2k} 2^{-2k-2} \frac{\Gamma(j) \Gamma(s-n/2+j)}{\Gamma(j+k+1) \Gamma(s-n/2+j+k+1)} \Delta_F^{k}.
$$
The resulting error term is given by the operator
$$
\cB_{j,N}(s) =  2^{-2N+2j-2} \frac{\Gamma(j)\Gamma(s-n/2+j)}{\Gamma(N+1) \Gamma(s-n/2+N+1)} \Delta_F^{N-j+1}.
$$

In Lemma~\ref{beta.lemma} we derive the (crude but uniform) beta function bound,
$$
\abs{\frac{\Gamma(k) \Gamma(\beta)}{\Gamma(\beta+k)}} \le C_\vep 2^{k},
$$
for $k\in \bbN$ and $\beta\in\bbC$ with $d(\beta, -\bbN_0) \ge \vep$.  This yields a uniform estimate
\begin{equation}\label{beta.est}
\abs{2^{-2k-2} \frac{\Gamma(j) \Gamma(s-n/2+j)}{\Gamma(j+k+1) \Gamma(s-n/2+j+k+1)}}  
\le C_\vep \Gamma(k+1)^{-2},
\end{equation}
for $d(s-n/2, -\bbN) \ge \vep$ and $j \in \bbN$.  

We can apply the assumption \eqref{df.quasi} along with \eqref{beta.est} to estimate $\abs{\pl^\alpha \cA_{j,N}(s) f}$ by
\[ \abs{\pl^\alpha \cA_{j,N}(s) f}\leq C^{N+\abs{\alpha}}N^{\abs{\alpha}}\sum_{k=1}^{N-1}(N/k)^{2k}\]
(the worst case being $j=1$).
The last term on the right is easily bounded above by $C^{N+\abs{\alpha}}N^{\abs{\alpha}}$ (after changing $C$ using our general convention).
It follows that $\cA_{j,N}(s)f$ satisfies the estimate \eqref{df.quasi} with an adjusted constant.
Clearly the same argument applies to $\cB_{j,N}(s) f$.
\end{proof}

\section{Special function estimates}

\subsection{Uniform beta function bounds}  The estimates we need on the beta function follow fairly directly from Stirling's formula, but they require a level uniformity beyond the standard results (which typically omit certain sectors to give sharper asymptotics).  Since this uniformity is crucial for us, we will give some details on these estimates.

\begin{lemma}\label{beta.lemma}
For $k \ge 1$ and $z \in \bbC$ we have the bounds
\[
\log \abs{\frac{\Gamma(z)\Gamma(k)}{\Gamma(z+k)}}  \le k\log 2 +  \log \bigl[1+ \dist(z, -\bbN_0)^{-1}\bigr] + \mc{O}(1),
\]
and
\[
\log \abs{\frac{\Gamma(z+k)}{\Gamma(z)\Gamma(k)}} \le (k + \abs{z}) \log 2 
+ \frac{\pi}2 \abs{\im z}  + \log \bigl[1+ \dist(z,-k -\bbN_0)^{-1}\bigr] + \mc{O}(\log (\abs{z}+k)).
\]
\end{lemma}

\begin{proof}
For $\re z \ge 0$ the application of Stirling's formula is direct and gives
\[
\begin{split}
\log \abs{\frac{\Gamma(z+k)}{\Gamma(z)\Gamma(k)}} &= (\re z) \log \abs{1 + \frac{k}{z}} + k \log \abs{1 + \frac{z}{k}} 
+ (\im z) [\arg(z) - \arg(z+k)] \\
&\qquad + \frac12 \log\abs{\frac{zk}{z+k}}+ \mc{O}(1).
\end{split}
\]
For the first two terms on the right we note that 
\[
0 \le (\re z) \log \abs{1 + \frac{k}{z}} + k \log \abs{1 + \frac{z}{k}} \le  \abs{z} \log \left(1 + \frac{k}{\abs{z}}\right)
+ k \log \left(1 + \frac{\abs{z}}{k}\right).
\]
For $x>0$, $y>0$, consider the real function, 
\[
f(x,y) := x \log \left(1 + \frac{y}{x}\right) + y \log \left(1 + \frac{x}{y}\right).
\]
Its graph is easily seen to lie below the tangent plane at $\{x=y\}$, which implies an estimate
\begin{equation}\label{fxy.bd}
f(x,y) \le (x+y) \log(2) \, .
\end{equation}
Hence we can estimate
\[
(\re z) \log \abs{1 + \frac{k}{z}} + k \log \abs{1 + \frac{z}{k}} \le (k + \abs{z}) \log(2) \, .
\]
Uniform estimates of the other terms are straightforward: for $\re z \ge 0$,
\[
0 \le  (\im z) [\arg(z) - \arg(z+k)] \le \frac{\pi}2 \abs{\im z},
\]
and
\[
- \log \left( 1 + \frac{1}{\abs{z}}\right) \le \log\abs{\frac{zk}{z+k}} \le \log \frac{\abs{z}+k}4.
\]
The resulting estimate for $\re z \ge 0$ is
\[
- \frac12 \log \left( 1 + \frac{1}{\abs{z}}\right) \le \log \abs{\frac{\Gamma(z+k)}{\Gamma(z)\Gamma(k)}} \le (k + \abs{z}) \log 2 
+ \frac{\pi}2 \abs{\im z} + \mc{O}(\log (\abs{z}+k)).
\]

For the remaining estimates we must use the reflection formula $\Gamma(z)\Gamma(1-z) = \pi/\sin \pi z$, which gives, for $\re z \le 0$,
\[
\begin{split}
\log |\Gamma(z)| &\le (\re z - \tfrac12) \log \abs{z} - \re z - \pi \abs{\im z} - \im z \arg(-z)\\
&\qquad + \log \bigl[1+ \dist(z, -\bbN_0)^{-1}\bigr] +  \mc{O}(1) ,
\end{split}
\]
and
\[
\log |\Gamma(z)| \ge (\re z - \tfrac12) \log \abs{z} - \re z - \pi \abs{\im z} - \im z \arg(-z) + \mc{O}(1).
\]
We obtain the estimates for $\re z \le -k$ by analyzing the terms just as above:
\[
\log \abs{\frac{\Gamma(z)\Gamma(k)}{\Gamma(z+k)}} \le  \log \bigl[1+ \dist(z, -\bbN_0)^{-1}\bigr]  + \mc{O}(1),
\]
and
\[
\log \abs{\frac{\Gamma(z+k)}{\Gamma(z)\Gamma(k)}} \le  (k+\abs{z}) \log 2 + \frac{\pi}2 \abs{\im z} + \log \bigl[1+ \dist(z,-k -\bbN_0)^{-1}\bigr] + \mc{O}(\log k).
\]
(In this case, the first term on the right comes from an application of \eqref{fxy.bd} to $f(\abs{z+k},k)$, plus the fact that $\abs{z+k} \le \abs{z}$ since $\re z\le -k$.)

For the case $-k \le \re z \le 0$, the bound 
\[
\log \abs{\frac{\Gamma(z+k)}{\Gamma(z)\Gamma(k)}} \le \frac{\pi}2 \abs{\im z} + \log \bigl[1+ \abs{z+k}^{-1} \bigr] + \mc{O}(\log \brak{z}),
\]
is obtained just as above.  The corresponding lower bound is slightly more delicate.  For $-k \le \re z \le 0$,
we have
\[
\log \abs{\frac{\Gamma(z)\Gamma(k)}{\Gamma(z+k)}} 
\le (-\re z) \log \abs{\frac{z+k}{z}} + k \log\abs{\frac{k}{z+k}} + \bigl[1+ \dist(z, -\bbN_0)^{-1}\bigr] + \mc{O}(1) 
\]
First we can estimate by restricting $z$ to the real axis:
\[
(-\re z) \log \abs{\frac{z+k}{z}} + k \log\abs{\frac{k}{z+k}} \le \abs{\re z} \log \left(\frac{k-\abs{\re z}}{\abs{\re z}}\right)
+ k \log \left(\frac{k}{k-\abs{\re z}}\right),
\]
for $-k \le \re z \le 0$.  Then we apply \eqref{fxy.bd} to $f(k-\abs{\re z},\abs{\re z})$.  The resulting estimate is
\[
\log \abs{\frac{\Gamma(z)\Gamma(k)}{\Gamma(z+k)}} 
\le k\log 2 +  \log \bigl[1+ \dist(z, -\bbN_0)^{-1}\bigr] + \mc{O}(1).
\]
\end{proof}

\subsection{Bessel function estimates}
To estimate model cusp terms, we need bounds on the modified Bessel functions that go slightly beyond the classical estimates, in that we require dependence of the constants for all complex values of the parameter.  

\begin{lemma}\label{K.lemma}
For $\lambda \in \bbC$ with $\abs{\lambda}>\vep$ and $x > 0$ we have
\[
\abs{K_\lambda(x)} \le \begin{cases} e^{c \abs{\re\lambda}} \max(1, \abs{\re\lambda}/x)^{\abs{\re\lambda}} e^{-x} & x\ge 1, \\
C_\vep \, \abs{\re\lambda}^{\abs{\re\lambda}} e^{c \abs{\re\lambda}} x^{-\abs{\re\lambda}} & x \le 1.  \end{cases}
\]
\end{lemma}
\begin{proof}
First note that $K_{-\lambda}(x) = K_{\lambda}(x)$, and that for any $x>0$,
\[
\abs{K_{\lambda}(x)} \le K_{\re \lambda}(x).
\]
Thus we can generally reduce to the case $\lambda = \sigma \ge 0$. 

For $x\ge 1$, we can cite the estimate from Paltsev \cite{Paltsev:1999}, which gives 
\[
K_\sigma(x) \asymp (x^2+\sigma^2)^{-1/4} \exp\left[-\sqrt{x^2+\sigma^2} + \sigma \log \frac{\sigma + \sqrt{x^2+\sigma^2}}{x} \right],
\]
for $\sigma \ge 0$, with constants independent of $x$ and $\sigma$.  Our estimate for $x \ge1$ follows from the simple observation that
\[
\frac{\sigma + \sqrt{x^2+\sigma^2}}{x}  \le (1+\sqrt{2}) \max(1, \sigma/x).
\]

For $0<x<1$ we can estimate from the standard integral form,
\[
K_\sigma(x) =  \frac{\sqrt{\pi}}{\Gamma(\sigma+\frac12)} (x/2)^{\sigma} \int_1^\infty e^{-xu}
(u^2-1)^{\sigma-\frac12}\>du,
\]
which holds for $\sigma > -1/2$.  Rescaling the integral gives
\[
\int_1^\infty e^{-xu} (u^2-1)^{\sigma-1/2}\>du = x^{-2\sigma} \int_x^\infty e^{-u} (u^2-x^2)^{\sigma-\frac12}\>du
\le x^{-2\sigma} \Gamma(2\sigma),
\]
for $\sigma>0$.  For $\sigma >\vep$ the estimate follows from Stirling's formula, and as noted above this covers the case $\re\lambda > \vep$.  We can extend the estimate for $0<x<1$ to the range $0 \le \re \lambda < \vep$, with $\abs{\lambda} \ge \vep$, using the identity
\[
K_\lambda(x) = \frac{x}{2\lambda} (K_{\lambda+1}(x) - K_{1-\lambda}(x)).
\]
\end{proof}

\begin{lemma}\label{I.lemma}
For $x>0$ and $\re\lambda\ge 0$ we have the estimates
$$
\abs{I_\lambda(x)} \le \begin{cases} e^{c\abs{\lambda}} \min(1, x/\re\lambda)^{\re \lambda} e^{x}  &
x\ge 1 \\
C e^{c\abs{\lambda}} \abs{\lambda}^{-\re\lambda} x^{\re\lambda} & x \le 1. \end{cases}
$$
\end{lemma}
\begin{proof}
First consider the case $x \ge 1$.  In this domain we again cite \cite{Paltsev:1999} for the estimate
\[
I_\sigma(x) \asymp (x^2+\sigma^2)^{-1/4} \exp\left[\sqrt{x^2+\sigma^2} + \sigma \log \frac{x}{\sigma + \sqrt{x^2+\sigma^2}} \right],
\]
for $\sigma \ge 0$.  In particular, this gives
$$
I_\sigma(x) \le e^{x} e^{c\sigma} \min(1, x/\sigma)^{\sigma},
$$
for $\sigma>0$, $x\ge 1$.  The general estimate for $\re \lambda \ge 0$, $x\ge 1$ then follows from 
$$
\abs{I_\lambda(x)} \le \frac{\Gamma(\re\lambda+\frac12)}{\abs{\Gamma(\lambda+\frac12)}} I_{\re\lambda}(x)
\le C e^{(\pi/2)\im\lambda}I_{\re\lambda}(x).
$$

For $0<x<1$ we work from the integral formula
$$
I_\lambda(x) = \frac{2}{\sqrt{\pi} \Gamma(\lambda+\frac12)} (x/2)^{\lambda} \int_0^1 \cosh(xu) 
(1-u^2)^{\lambda-\frac12}\>du,
$$
valid for $\re \lambda>-1/2$.  For $x<1$ and $\re \lambda \ge 0$ we can simply bound the integral by a constant, 
and so from Stirling's formula we obtain the estimate
$$
\abs{I_\lambda(x)} \le C e^{c\abs{\lambda}} \abs{\lambda}^{-\re\lambda} x^{\re\lambda}.
$$
\end{proof}

For the cusp resolvent estimates, we need to apply these Bessel function bounds to the function,
\[
F_{s,x,x'}(\tau) := K_\lambda(x\tau) I_\lambda(x'\tau) H(x-x') + I_\lambda(x\tau) K_\lambda(x'\tau) H(x'-x),
\]
where $\lambda := s-n/2$.

\begin{lemma}\label{F.est.lemma}
There exist constants $c>0,C>0$ such that, for all $s\in \bbC$, 
\begin{equation}\label{Fs.est}
\begin{split}
\abs{F_{s,x,x'}(\tau)} & \le   e^{c\abs{\lambda}} \max\left(\abs{\re\lambda}^{-2\re\lambda}, 1\right) \\
&\qquad \x\begin{cases}
C & \text{both }x\tau, x'\tau \ge 1, \\
C\max((xx'\tau^2)^{\re\lambda},1) & \text{both }x\tau, x'\tau \le 1, \\
C\max((x\tau)^{\re\lambda},1) & x\tau < 1 < x'\tau, \\
C\max((x'\tau)^{\re\lambda},1) & x'\tau < 1 < x\tau.\\
\end{cases}
\end{split}
\end{equation}
\end{lemma}

\begin{proof}
Immediately from Lemmas \ref{K.lemma} and \ref{I.lemma} we have, for $\re \lambda \ge 0$,
\[
\abs{F_{s,x,x'}(\tau)} \le C e^{c |\lambda|}
\begin{cases}
\min(x/x',x'/x)^{1/2} & \text{both }x\tau, x'\tau \ge 1 \\
\min(x/x',x'/x)^{\re\lambda} & \text{both }x\tau, x'\tau \le 1 \\
(x/x')^{\re\lambda} & x\tau < 1 < x'\tau \\
(x'/x)^{\re\lambda} & x'\tau < 1 < x\tau. \\
\end{cases}
\]

To extend the estimates to $\re \lambda \le0$,  we use the identity,
$$
I_{-\lambda}(z) = I_\lambda(z) + \frac{2\sin \pi \lambda}{\pi} K_\lambda(z).
$$
The bounds on $F$ work as before, except for the new term 
\[
\begin{split}
&\abs{\frac{2\sin \pi \lambda}{\pi} K_\lambda(x\tau) K_\lambda(x'\tau)} \\
&\qquad \le
C \abs{\re\lambda}^{2\abs{\re\lambda}} e^{c\abs{\lambda}}
\begin{cases}
\min(x/x',x'/x) & \text{both }x\tau, x'\tau \ge 1, \\
(xx'\tau^2)^{-\abs{\re\lambda}} & \text{both }x\tau, x'\tau \le 1, \\
(x\tau)^{-\abs{\re\lambda}} & x\tau < 1 < x'\tau, \\
(x'\tau)^{-\abs{\re\lambda}} & x'\tau < 1 < x\tau. \\
\end{cases}
\end{split}
\]
For $\re \lambda \le 0$ this new term dominates the estimate of $F$, yielding the general estimate \eqref{Fs.est}.
\end{proof}


\end{document}